\documentclass[11pt,reqno]{amsart}

\usepackage{amssymb}
\usepackage{amscd}
\usepackage{amsfonts}
\usepackage{mathrsfs}
\usepackage{setspace}
\usepackage{version}
\usepackage{mathtools}
\usepackage[pdftex,colorlinks]{hyperref}

\usepackage{xcolor}

\numberwithin{equation}{section} \theoremstyle{plain}
\newtheorem{theorem}[subsection]{Theorem}
\newtheorem{proposition}[subsection]{Proposition}
\newtheorem{lemma}[subsection]{Lemma}
\newtheorem{corollary}[subsection]{Corollary}
\newtheorem{definition}[subsection]{Definition}

\newtheorem{remark}[subsection]{Remark}

\renewcommand{\leq}{\leqslant}
\renewcommand{\geq}{\geqslant}
\newsavebox{\proofbox}
\savebox{\proofbox}{\begin{picture}(7,7)  \put(0,0){\framebox(7,7){}}\end{picture}}

\newcommand\E{\mathbb{E}}
\newcommand\Z{\mathbb{Z}}
\newcommand\R{\mathbb{R}}

\newcommand\C{\mathbb{C}}
\newcommand\N{\mathbb{N}}

\newcommand\HH{\mathbb{H}}
\newcommand\SL{\operatorname{SL}}

\newcommand\Cay{\operatorname{Cay}}

\newcommand\GL{\operatorname{GL}}
\newcommand\PGL{\operatorname{PGL}}

\renewcommand\P{\mathbb{P}}
\newcommand\F{\mathbb{F}}
\newcommand\G{\mathbb{G}}
\newcommand\Q{\mathbb{Q}}
\newcommand\X{\mathbf{X}}

\newcommand\eps{\varepsilon}

\begin{document}

\title{Approximate subgroups and super-strong approximation}
\author{Emmanuel Breuillard}
\address{Laboratoire de Math\'ematiques\\
B\^atiment 425, Universit\'e Paris Sud 11\\
91405 Orsay\\
FRANCE}
\email{emmanuel.breuillard@math.u-psud.fr}


\begin{abstract}
Surveying some of the recent developments on approximate subgroups and super-strong approximation for thin groups, we describe the Bourgain-Gamburd method for establishing spectral gaps for finite groups and the proof of the classification of approximate subgroups of semisimple algebraic groups over finite fields. We then give a proof of the super-strong approximation for mod $p$ quotients via random matrix products and a quantitative version of strong approximation. Some applications to the group sieve are also presented. These notes are based on a series of lectures given at the 2013 Groups St. Andrews meeting.
\end{abstract}

\maketitle
\setcounter{tocdepth}{1}
\tableofcontents

\section{Introduction}

In the early 1980's  Matthews-Vaserstein-Weisfeiler \cite{MVW}, and then Nori \cite{nori} and Weisfeiler \cite{weisfeiler} (independently) proved the following theorem:

\begin{theorem}[Strong-approximation theorem]\label{strongapp} Suppose $\G$ is a connected, simply connected, semisimple algebraic group defined over $\Q$, and let $\Gamma \leq \G(\Q)$ be a finitely generated Zariski-dense subgroup. Then for all sufficiently large prime numbers $p$, the reduction $\Gamma_p$ of $\Gamma$ is equal to $\G_p(\F_p)$.
\end{theorem}

For example, if $\Gamma \leq \SL_n(\Z)$ is a finitely generated Zariski dense subgroup, then $\Gamma_p=\SL_n(\Z/p\Z)$ for all large enough prime numbers $p$. When $p$ is large enough, the algebraic group $\G$ (viewed as a closed subgroup of some $\GL_n$) admits a smooth reduction defined over $\F_p$, which we denote by $\G_p$.
 Since $\Gamma$ is finitely generated, there are finitely many primes $p_1,\ldots,p_k$ (appearing in the denominators of the matrix entries of $S$) such that $\Gamma$ belongs to $\G(\Z[\frac{1}{p_1},\ldots,\frac{1}{p_k}]):= \G \cap \GL_n(\Z[\frac{1}{p_1},\ldots,\frac{1}{p_k}])$, and the reduction modulo $p$ map is well-defined on this subgroup if $p$ is large enough.

The result fails if $\G$ is not simply connected (e.g. the image of $\SL_2(\Z)$ in $\PGL_2(\F_p)$ has index $2$ when $p>2$). However every connected absolutely \ almost simple algebraic group admits a simply connected finite cover to which we can lift $\Gamma$ and apply the theorem. This yields that $[\G_p(\F_p):\Gamma_p]$ is nevertheless always bounded (for $p$ large) by a constant depending only on $\G$ (one can take $1+rank(\G)$, see \cite[Remark 3.6]{nori}).

A similar result holds for groups defined over number fields instead of $\Q$. Its proof reduces to the case of $\Q$ by suitable restriction of scalars. See Remark \ref{numberfields} below (see also \cite{weisfeiler}).

That the result holds in the case when $\Gamma$ is an $S$-arithmetic group $\Gamma=\G(\Z[\frac{1}{p_1},\ldots,\frac{1}{p_m}])$ was known much earlier by work of Kneser \cite{kneser} and Platonov \cite{platonov} in particular. See \cite[Chapter 7]{platonov-rapinchuk}) and \cite{rapinchuk-survey}.

Theorem \ref{strongapp} is then of particular interest when the group $\Gamma$ is not a full $S$-arithmetic subgroup of $\G$ but has infinite index in one of them, while still remaining Zariski dense in $\G$ ($S$-arithmetic subgroups are Zariski dense by the Borel density theorem). Such a group is called a \emph{thin subgroup} of $\G$ in recent terminology due to Peter Sarnak \cite{sarnak-notes}.

\bigskip

What we call \emph{super-strong approximation} is the fact stated in Theorem \ref{ssapp} below that $\Gamma$ not only surjects onto $\G_p(\F_p)$ for $p$ large but that the associated Cayley graphs of $\G_p(\F_p)$ form a \emph{family of expanders}. The goal of these notes is to give a proof of this fact, give some applications, and introduce the reader to the various techniques used in the proof.

\bigskip

It is of course not the purpose of this survey to give a complete introduction to expander graphs and for that matter we refer the reader to the many sources on the subject starting with Lubotzky's monograph \cite{lubotzky} and survey \cite{lubotzky-survey} (see also  \cite{hoory-linial-wigderson} and \cite{kowalskinotes,tao-forthcoming-book,breuillard-PCMI}). Let us simply recall that to every finite $k$-regular graph $\mathcal{G}$ is associated a combinatorial Laplace operator acting on the (finite dimensional) space of functions on the vertices of the graph. It is defined by the formula

$$\Delta f (x) = f(x) - \frac{1}{k} \sum_{y \sim x} f(y),$$
where $y\sim x$ is a vertex connected to $x$ by an edge. This operator is symmetric and non-negative. Its eigenvalues are real and non-negative. The eigenvalue $0$ comes with multiplicity one if the graph is connected and the first nonzero eigenvalue is denoted by $\lambda_1(\mathcal{G})$ and satisfies satisfies:

\begin{equation}\label{lambdadef}
 \lambda_1(\mathcal{G}) = \inf \{ \langle \Delta f, f\rangle , ||f||_2=1, \sum_x f(x) = 0\}.
 \end{equation}

 An infinite family of $k$-regular graphs $(\mathcal{G}_n)_{n \geq 1}$ is said to be a \emph{family of expanders} if there is $\eps>0$ such that for all $n\geq 1$,

$$\lambda_1(\mathcal{G}_n) > \eps.$$

We are now in a position to state the following strengthening of Theorem \ref{strongapp}.

\begin{theorem}[Super-strong approximation]\label{ssapp} Suppose $\G$ is a connected, simply connected, semi-simple algebraic group defined over $\Q$, and let $\Gamma \leq \G(\Q)$ be a Zariski-dense subgroup generated by a finite set $S$. Then there is $\eps=\eps(S)>0$ such that for all large enough prime numbers $p$, the reduction $\Gamma_p$ of $\Gamma$ is equal to $\G_p(\F_p)$ and the associated Cayley graph $\Cay(\G_p(\F_p),S_p)$ is an $\eps$-expander.
\end{theorem}

Here $S_p$ is the image of $S$ by reduction modulo $p$. As before, the result also holds if $\G$ is not assumed to be simply connected, but $\Gamma_p$ may then only be a subgroup of $\G_p(\F_p)$ whose index is nevertheless bounded independently of $p$, while $\Cay(\Gamma_p,S_p)$ remains an $\eps$-expander.

\bigskip

This theorem is a special case of a result due to Salehi-Golsefidy and Varj\'u \cite{salehi-golsefidy-varju}, which asserts that the conclusion  also holds for quotient modulo a square free integer and even when the connected algebraic group $\G$ is only assumed to be perfect. Their proof follows the so-called Bourgain-Gamburd expansion machine, which can be implemented in this context in part thanks to the recent results on approximate subgroups of linear groups due to Pyber-Szab\'o \cite{pyber-szabo} and Breuillard-Green-Tao \cite{breuillard-green-tao-linear}.

\bigskip

In these notes we describe the Bourgain-Gamburd method as well as the above mentioned results on approximate subgroups and finally give a complete proof of Theorem \ref{ssapp} (i.e. of super-strong approximation for mod $p$ quotients) following a somewhat alternate route than in \cite{salehi-golsefidy-varju} by use of random matrix products \cite{breuillard-large-deviations}.

\subsection{The Lubotzky alternative and its expander version}

One can formulate a version of the strong approximation theorem, which is valid for every finitely generated subgroup of $\GL_d(k)$, where $k$ is an arbitrary field of characteristic zero (one can also deal with the positive characteristic case thanks to the work of Pink \cite{pink}, however no super-strong version is known in positive characteristic thus far). When the group $\Gamma=\langle S \rangle$ we start with is non virtually solvable, one can show that there is a non trivial connected and simply connected semisimple algebraic group $G$ defined over $\Q$ and a group homomorphism from a finite index subgroup of $\Gamma$ into $\G(\Q)$ with a Zariski-dense image (see \cite[Prop. 16.4.13]{lubotzky-segal} and the discussion that follows). This allows to then apply the strong-approximation theorem \ref{strongapp} and deduce that $\Gamma_0$ admits $\G_p(\F_p)$ as a quotient for almost all $p$.

This information was used in a key way by  Lubotzky and Mann in their work on subgroup growth \cite{lubotzky-mann}. For this version of strong approximation, called \emph{the Lubotzky alternative}, and we refer the reader to the notes devoted to it and its various refinements in the book by Lubotzky and Segal on subgroup growth (\cite[16.4.12]{lubotzky-segal}, see also \cite{klopsch-nikolov-voll}). Strengthened by the super-strong approximation theorem, this gives the following statement:

\begin{theorem}(Lubotzky super-alternative)\label{lubalt} Let $S$ be a finite symmetric subset of $\GL_d(k)$, where $k$ is a field of characteristic zero. Then the subgroup $\Gamma=\langle S \rangle$ generated by $S$ contains a subgroup $\Gamma_0$ whose index $m$ in $\Gamma$ is finite and bounded in terms of $d$ only, such that
\begin{itemize}
\item either the subgroup $\Gamma_0$ is solvable,
 \item or there is a connected, simply connected, semisimple algebraic group $\G$ defined over $\Q$, such that for all large enough primes $p \in \N$, there is a surjective group homomorphism $\rho_p$ from $\Gamma_0$ to $\G_p(\F_p)$ such that the Cayley graph $\Cay(\G_p(\F_p), \rho_p(S_0))$ is an $\eps$-expander, for some $\eps>0$ independent of $p$, where $S_0$ is a subset of $S^{2m}$ generating $\Gamma_0$.
     \end{itemize}
\end{theorem}

Note that given a group $\Gamma$ generated by a symmetric set $S$, then every subgroup of finite index  $\Gamma_0$ is finitely generated by a symmetric subset contained in $S^{2m-1}$, if $m$ is the index of $\Gamma_0$ in $\Gamma$ (e.g. see \cite[Lemma C.1]{breuillard-green-tao-linear}).

\bigskip

A version of Theorem \ref{lubalt} for a bounded number of primes is also true: given large enough distinct primes $p_1,\ldots,p_k$, the Cayley graphs $\Cay(G(\F_{p_1})\times \ldots \times G(\F_{p_k}) , (\rho_{p_1}\times \ldots \times \rho_{p_k})(S))$ are $\eps$-expanders for a uniform $\eps>0$ independent of the number of primes $k$. We will prove this stronger version only with an $\eps$ depending on $k$ (but not on the choice of $k$ primes). See Theorem \ref{severalp} below. One needs the works of Varj\'u \cite{varju} and Salehi-Golsefidy-Varj\'u \cite{salehi-golsefidy-varju} to get this uniformity in the number of primes, but the proof is rather more involved. Note that at any case $\eps$ depends on $S$ and it is an open question whether this dependence can be removed (see \cite{breuillard-gamburd} for partial results in this direction).

\subsection{The group sieve method}

Knowing that the finite quotients Cayley graphs are expanders is a very useful information for a number of applications to group theory and number theory, in particular it is the basis of the so-called Group Sieve, pioneered by Kowalski \cite{kowalski-book, kowalski-bourbaki}, Rivin \cite{rivin}, and Lubotzky-Meiri \cite{lubotzky-meiri, lubotzky-meiri2} and of the Affine Sieve of Bourgain-Gamburd-Sarnak \cite{bourgain-gamburd-sarnak}. See \cite{kontorovich} and \cite{kowalski-notes} for two nice expositions.

 Roughly speaking, the expander property allows to give very good bounds on the various error terms that appear when sieving modulo primes. In these notes, we will give a general statement, \emph{the group sieve lemma} (Lemma  \ref{group-sieve} below), due to Lubotzky and Meiri, which allows to show that a subset $Z$ of a given finitely generated linear group is exponentially small, provided its reduction modulo $p$ does not occupy too large a subset of the quotient group for many primes $p$. For this version of the group sieve, expansion for pairs of primes is sufficient (i.e. we need that $G(\F_{p_1}) \times G(\F_{p_2})$ expands for $p_1 \neq p_2$), so our version of the Lubotzky super-alternative above will be enough. Expansion for all square free moduli is necessary however, and sometimes crucial, in other situations, such as in the Affine Sieve pioneered by Bourgain-Gamburd-Sarnak \cite{bourgain-gamburd-sarnak} and further developed by Salehi-Golsefidy-Sarnak \cite{salehi-sarnak}, Bourgain and Kontorovich \cite{bourgain-kontorovich} and others.

\bigskip

The conclusion of the super-strong approximation theorem (Theorem \ref{ssapp}) can be reformulated in the following way: there is $\eps>0$ depending only on the generating set $S$ such that for every real valued function $f$ on the group $\G_p(\F_p)$, such that $\sum_{x \in \G_p(\F_p)} f(x)=0$ and $||f||^2_{\ell^2}= \sum_{x \in \G_p(\F_p)} |f(x)|^2=1$,
$$\langle \Delta f,f \rangle > \eps,$$
where
$$\langle \Delta f,f \rangle = \frac{1}{2k}\sum_{s \in S} ||s \cdot f - f||^2_{\ell^2} = \frac{1}{2k}\sum_{s \in S}  \sum_{x \in \G_p(\F_p)} |f(s^{-1}x) -f(x)|^2.$$
Let $S_p=\{s_1,\ldots,s_k\}$ be the image of $S$ under the reduction modulo $p$ map and $\mu_{S_p}$ be the uniform probability measure on $S_p$, assigning equal mass $1/k$ ($=\frac{1}{|S|}$ for $p$ large enough) to each element of $S_p$.

$$\mu_{S_p} := \frac{1}{k}(\delta_{s_1} + \dots + \delta_{s_k})$$

Note that $\mu_{S_p}=Id - \Delta$ as operators on $\ell^2(\G_p(\F_p))$, and hence its operator norm on $\ell^2_0(\G_p(\F_p))$, the orthogonal of constants, satisfies:

$$||\mu_{S_p}|_{\ell^2_0}||< 1- \eps $$

It is in this form that the theorem is used in its applications to the group sieve method. For example it allows Lubotzky and Meiri \cite{lubotzky-meiri} to establish the following result about the scarcity of proper powers in non virtually solvable linear groups. A group element is called a proper power if it is of the form $g^n$ for some integer $n \geq 2$ and some other group element $g$ (from the same group).

\begin{theorem}(Lubotzky-Meiri \cite{lubotzky-meiri}) Let $\Gamma \leq \GL_d(\C)$ be a finitely generated subgroup and let $\mu_S$ be the uniform probability measure on a finite symmetric generating $S$. Assume that $\Gamma$ is not virtually solvable. Then the set $\mathcal{P}_\Gamma$ of proper powers in $\Gamma$ is exponentially small in the sense that there is $c=c(S)>0$ such that for every $n \in \N$,
$$\mu_S^n (\mathcal{P}_\Gamma) \leq e^{-c n}.$$
\end{theorem}

Here $\mu_S^n$ is the $n$-th convolution power of the probability measure $\mu_S$ on $\Gamma$. Equivalently, it is the distribution at time $n$ of the simple random walk starting at the identity on the associated Cayley graph $\Cay(\Gamma,S)$. Or more explicitly:

$$\mu_S^n (\mathcal{P}_\Gamma) = \P_{w \in W_{n,k}}(\mathcal{P}_\Gamma):=\frac{|\{ w, |w|= n , \overline{w} \in \mathcal{P}_\Gamma\}|}{|\{ w, |w|= n\}|},$$
where $W_{n,k}$ is the set of (non reduced!) words $w$ of length $|w|=n$ in the formal alphabet made of letters from the set $S$, and $\overline{w}$ its value as a group element when computed inside $\Gamma$. One can analogously count reduced words of length $n$ in the free group and get the same result, but we note in passing that obtaining a result of this kind for the average with respect to the word metric on $\Gamma$ induced by $S$ seems out of reach at the moment, because little is known about the balls for the word metric on a group of exponential growth.

\bigskip

\subsection{On the proof of the super-strong approximation theorem} Theorem \ref{ssapp} was first proved in the special case of subgroups of $\SL_2(\Z)$ in a remarkable breakthrough by Bourgain and Gamburd \cite{bourgain-gamburd}. They deduced the expansion by showing that the simple random walk on the finite quotient $\SL_2(\Z/p\Z)$ must equidistribute very fast, indeed after only $O(\log p)$ steps. In doing so they reversed the traditional way of looking at things: traditionnally spectral gaps estimates were proven by other methods (e.g. representation theory, property $(T)$, etc.) and were then used to prove fast equidistribution of random walks. Bourgain and Gamburd reversed this order, first proving equidistribution and then deducing the gap (see Proposition \ref{rwchar} below for the equivalence between spectral gap and fast equidistribution).

This idea can be traced back to the seminal work of Sarnak and Xue \cite{sarnak-xue}, which gave a new, softer,  approach toward Selberg's $3/16$ theorem (i.e. the first eigenvalue of the Laplace operator on quotients of the hyperbolic plane by congruence subgroups of $SL(2,\Z)$ is at least $3/16$, see \cite{selberg}). They  exploited, via the trace formula, the high multiplicity of the spectrum coming from the $\frac{p-1}{2}$ lower bound on the dimension of the smallest non trivial complex representation of $\SL_2(\F_p)$ (this bound goes back to Frobenius) and a soft combinatorial upper bound on the number of lattice points in a ball of radius roughly $\log p$. We refer the reader to the expository papers of P. Sarnak \cite{sarnak-notices, sarnak-icm}, where this method and its history (in particular the role of Bernstein and Kazhdan) is described.

In his thesis \cite{gamburd-thesis} Gamburd pursued this method and established the first spectral gap result valid for thin groups: he showed that if a finitely generated subgroup $\Gamma$ of $\SL_2(\Z)$ is large enough in the sense that the Hausdorff dimension of its limit set on $\P^1(\R)$ is at least $\frac{5}{6}$, then the spectrum of the associated (infinite volume) quotients of the hyperbolic plane modulo the congruence subgroups $\Gamma_p := \Gamma \cap \ker ( \SL_2(\Z) \to \SL_2(\Z/p\Z))$  admits a uniform lower bound independent of $p$. In turn the resulting Cayley graphs of $\SL_2(\Z/p\Z)$ are expander graphs.

Bourgain and Gamburd \cite{bourgain-gamburd} pushed the method even further to implement it for all Zariski-dense subgroups of $\SL_2(\Z)$ with no restriction on the limit set. The structure of their proof retained the same patterns, playing the high multiplicity lower bound against a combinatorial upper bound via the trace formula applied to convolution powers of a fixed probability measure on the generating set. Achieving this combinatorial upper bound is the gist of their work: they brought in an important graph theoretic result (the Balog-Szemer\'edi-Gowers lemma, a parent of the celebrated Szemer\'edi regularity lemma) revisited in this context by Tao \cite{tao-noncommutative} to show that convolution powers of probability measures decay in $\ell^2$ norm (the so-called $\ell^2$-flattening) unless the measure charges significantly a certain  approximate subgroup. That  there  exists no interesting approximate subgroup of $SL_2(\F_p)$ was established for this purpose by Helfgott \cite{helfgott}. The combinatorial upper bound (on the probability of return to the identity of the simple random walk at time roughly $\log p$), and hence the spectral gap,  then reduces to establishing a certain non concentration estimate on subgroups  for random walks on $\SL_2(\Z)$ (see Theorem \ref{deviation}), which in this case can easily be deduced from Kesten's theorem \cite{kesten}.

This new method became known as the \emph{Bourgain-Gamburd expansion machine} (see e.g. the papers \cite{bgt-suzuki, bggt} as well as the forthcoming book \cite{tao-forthcoming-book}). Its scope goes beyond $SL_2(\F_p)$ and, quite remarkably, it can potentially be applied to any finite group (see Proposition \ref{machine} for a precise formulation of the method and its ingredients). It was understood early on that the scheme of the proof in \cite{bourgain-gamburd} was general enough that it could be made to work in the general setting of Theorem \ref{ssapp}, provided one could establish each step in the right generality.  The bounds on the dimension of complex representations are well-known thanks to classical work of Landazuri-Seitz \cite{landazuri-seitz}. The graph theoretic lemma needs no modification in the general setting. The remaining two items however require deeper consideration. The classification of approximate groups, first established by Helfgott for $SL_2(\F_p)$ and $SL_3(\F_p)$, was finally completed in the general case by Pyber and Szab\'o \cite{pyber-szabo} and independently by Breuillard-Green-Tao \cite{breuillard-green-tao-linear}. Regarding the upper bounds on the probability of hitting a subgroup, there are two known ways to achieve them. The first is to use the theory of random matrix products, and this was done in subsequent work of Bourgain-Gamburd \cite{bourgain-gamburd-SLn}, but only in the special case of subgroups of $SL_n(\Z)$, because the estimates from the theory of random matrix products required to deal with the general case were lacking. The second consists in applying a ping-pong argument akin to the proof of the Tits alternative \cite{tits-alternative}, and this was performed by Varj\'u in his thesis \cite{varju} and subsequently by Salehi-Golsefidy and Varj\'u in their joint work \cite{salehi-golsefidy-varju}, in which they establish Theorem \ref{ssapp} in full generality.

In the remainder of these notes we will prove Theorem \ref{ssapp}  following each of these steps very closely. The only novelty in our proof lies in the last step: thanks to \cite{breuillard-large-deviations}, we now understand how to use random matrix products to prove in the desired generality the required upper bounds for the probability of hitting a subgroup (the non-concentration estimates). This approach is somewhat more direct than the one taken by Salehi-Golsefidy and Varj\'u in \cite{salehi-golsefidy-varju}, and it is very close to what Green, Tao and I had in mind, when we announced a proof of Theorem \ref{ssapp} in \cite[Theorem 7.3]{breuillard-green-tao-linear-announce} in the special case of absolutely simple groups over $\Z$, but never came to the point of writing it up in full.

As already mentioned Salehi-Golsefidy and Varj\'u \cite{salehi-golsefidy-varju} actually proved a strong version of Theorem \ref{ssapp} showing the expansion property also for the quotients modulo a square free integer, and assuming only that $\G$ is perfect (which is also a necessary condition for expansion). See Theorem \ref{alirezapeter} below. That strong version is crucial for certain applications to sieving in orbits (\`a la Bourgain-Gamburd-Sarnak \cite{bourgain-gamburd-sarnak}), but its proof is much more involved. Often it is enough to have Theorem \ref{ssapp}, or its extension to two or a bounded number of primes, which is not more costly. That will be the case for the applications presented in this paper. This, I thought, was enough justification for writing a complete proof of super-strong approximation for prime moduli in one place.

\subsection{Outline of the article}
In Section \ref{nori-strong} we present a proof of the strong approximation theorem of Matthews, Vassertein and Weisfeiler following Nori's proof. Our treatment yields a quantitative version in the sense that it gives a upper bound on the first $p$ for which the surjectivity of the reduction mod $p$ holds in terms of the height of the generating set. Section \ref{bg-sec} is devoted to the Bourgain-Gamburd machine: we state very general conditions on the Cayley graph of an arbitrary finite group that are sufficient to establish a spectral gap. Section \ref{approxsec} is devoted to approximate subgroups of linear groups over finite fields. We prove there the theorem of Pyber-Szab\'o and Breuillard-Green-Tao. In Section \ref{rmtsec} we discuss random matrix products and a general non-concentration on subgroups result for random walks on linear groups. Finally in Section \ref{pfsapp} we combine the results of the preceding three sections to complete the proof of the super-strong approximation theorem in the case of mod $p$ quotients (Theorems \ref{ssapp} and \ref{severalp}). The final section is devoted to applications to the group sieve method and results of Aoun, Jouve-Kowalski-Zywina, Lubotzky-Meiri, Lubotzky-Rosenzweig and Prasad-Rapinchuk on generic properties elements in non virtually solvable linear groups.

\section{Nori's theorem and a quantitative version of strong approximation}\label{nori-strong}

It was Matthews Vaserstein and Weisfeiler \cite{MVW} who first proved  the strong approximation theorem for Zariski-dense subgroups, i.e. Theorem \ref{strongapp}, in the case when $G$ is absolutely simple. Their proof  made use of the (brand new at the time) classification of finite simple groups. Another, classification-free proof was found roughly at the same time and independently by M. Nori, yielding also the case $G$ semisimple, as a consequence of the following general result proved in \cite{nori}.

\begin{theorem}[Nori \cite{nori}]\label{norithm} Let $H$ be a subgroup of $\GL_n(\F_p)$, and $H^{+}$ the subgroup generated by its elements of order $p$. If $p$ is larger than some constant $c(n)$ depending only on $n$, then there is a connected algebraic subgroup $\widetilde{H}$ of $\GL_n$ defined over $\F_p$ such that $H^+$ coincides with $\widetilde{H}(\F_p)^{+}$. Moreover there is a normal abelian subgroup $A \leq H$ such that $[H:AH^+]$ is bounded in terms of $n$ only.
\end{theorem}

Observe that if $p\geq n$, then elements of order $p$ in $\GL_n(\F_p)$ are precisely the unipotent matrices: indeed $x^p=1$ is equivalent to $(x-1)^p=0$ for $x \in \GL_n(\F_p)$ and hence to $x=1+n$, where $n$ is a nilpotent matrix.
As Nori explains in \cite[Remark 3.6.]{nori}, the index of $\widetilde{H}(\F_p)^{+}$ in $\widetilde{H}(\F_p)$ is bounded by a function of $n$ only. So the meaning of Nori's theorem is that finite subgroups of $\GL_n(\F_p)$ generated by elements of order $p$ are essentially algebraic subgroups, if $p>c(n)$.

The key feature of Nori's theorem is that no assumption whatsoever is made on the subgroup $H$. Hence Nori's theorem can be seen as a description of \emph{arbitrary} subgroups of $\GL_n(\F_p)$. It can be viewed as complementing the celebrated theorem of Camille Jordan \cite{jordan} on finite subgroups of $\GL_n(K)$ whose order is prime to the characteristic of the field $K$ : such a group admits an abelian subgroup whose index is bounded by some function of $n$ only. Nori's theorem explains what happens when the characteristic divides the order of the finite group: recall that a finite group has an element of prime order $p$ if and only if its order is a multiple of $p$ (Cauchy's theorem).

Jordan's theorem is usually quoted for subgroups of $\GL_n(\C)$, but this stronger version can be derived easily by lifting the group to $\C$ (see \cite[Theorem C]{nori}). In fact Jordan had already proved this stronger version in his original paper: his proof is purely algebraic and applies to any finite subgroup of $\GL_n(K)$ all of whose elements are semisimple (or equivalently to finite subgroups without a non trivial unipotent element), where $K$ is any algebraically closed field (see \cite{breuillard-jordan} for a discussion).

Textbooks presenting Jordan's theorem usually give a different, more geometric treatment, due to Frobenius,  Bieberbach and Blichfeldt. Jordan's own argument seems to have been forgotten for more than a hundred years until Larsen and Pink \cite{larsen-pink} rediscovered it and generalized it considerably to obtain a classification of all finite subgroups of $\GL_d$ in every characteristic. The Larsen-Pink theorem is more general than Nori's result stated above in that it applies to finite subgroups of $\GL_d$ regardless of the field and the size of the characteristic. We will comment on the Larsen-Pink theorem further below, when we discuss approximate subgroups of linear groups. The proof of the Larsen-Pink theorem, which by the way is also independent of the classification of finite simple groups, plays a key role in the structure theorem for approximate subgroups of linear groups (see Theorem \ref{class} below).

For the applications to strong and super-strong approximation, we will not need the full force of Theorem \ref{norithm} above. Rather the following important special case will be sufficient.

\begin{theorem}(Sufficiently Zariski-dense subgroups)\label{suffdense} There is $M=M(d)$ such that the following holds. Let $p>M$ be a prime number and $\G_p \leq \GL_d$ be a semisimple simply connected algebraic group defined over $\F_p$. If a subgroup $H \leq \G_p(\F_p)$ is not contained in a proper algebraic subgroup of $\G_p$ of complexity at most $M$, then it must be equal to $\G_p(\F_p)$.
\end{theorem}

We say informally that a closed algebraic subvariety of $\GL_d$ has complexity at most $M$ if it can be defined as the vanishing locus of a finite set of polynomials such that the sum of their degrees in each variable is at most $M$. See \cite{breuillard-green-tao-linear} for background on this notion. It is particularly useful in positive characteristic: saying that a finite subgroup of $\GL_d(\overline{\F_p})$ is algebraic is meaningless, because every finite subgroup is an algebraic subset with several (possibly many) irreducible components. However putting a bound on the complexity forces a bound on the number of irreducible components (\cite[Lemma A.4]{breuillard-green-tao-linear}) and hence restricts the class of finite subgroups drastically and leads to interesting statements, such as the above.

\bigskip

We now sketch Nori's proof of Theorem \ref{suffdense}. A similar argument is due to Gabber, see \cite[Thm 12.4.1]{katz}. Pushing this idea a bit further allows Nori to also prove Theorem \ref{norithm}.


\bigskip

\begin{proof}(sketch) If $H$ had no non trivial unipotent element, it would have an abelian subgroup of bounded index by Jordan's theorem. But this would violate the assumption that $H$ is sufficiently Zariski-dense. So $H$ contains a unipotent element, which we may write in the form $h=\exp \xi$, for some nilpotent matrix $\xi$. The $\F_p$-span $V_H$ of all $H$-conjugates of $\xi$ is invariant under the adjoint action of $H$. The assumption that $H$ is sufficiently Zariski-dense implies that $V_H$ must be the full $\F_p$-Lie algebra of $\G_p$ in $\emph{gl}_d(\F_p)$. Pick unipotent elements $h_1,\ldots,h_d \in H$ such that the corresponding $\xi_i$'s form a basis of $Lie(\G_p)$.

Now consider the map $\Phi: \F_p^{\dim \G} \to \G_p(\F_p)$, $(t_1,\ldots,t_d) \mapsto h_1^{t_1} \cdot \ldots \cdot h_d^{t_d}$. Note that $\Phi$ is a polynomial map whose degree is bounded in terms of $d$ only. Its image lies in $H$. We claim that there is a constant $c=c(d)>0$ such that $|Im \Phi| \geq c p^{d}$. Indeed, the jacobian of $\Phi$ is not identically zero, so outside its vanishing locus (a proper subvariety, hence a subset of size $O(p^{d-1})$) the fibers of $\Phi$ are of bounded cardinality. This implies the desired bound.

Now since there are positive constants $c_1,c_2$ such that $c_1p^d \leq |\G_p(\F_p)| \leq c_2 p^d$ (e.g. see \cite[Lemma 3.5.]{nori}), we get that the index $[\G_p(\F_p):H]$ is bounded. However since $\G$ is simply connected, $\G_p(\F_p)$ is an almost direct product of quasi-simple groups and thus has no subgroups of bounded index when $p$ is large (Kneser-Tits for $\F_p$, see \cite{platonov-rapinchuk}, see also Remark \ref{nosmallindex}). Hence $H=\G_p(\F_p)$.
\end{proof}

Nori's proof of strong approximation (i.e. of Theorem \ref{strongapp}) is based on Theorem \ref{suffdense} alone. We will explain this argument below. It turns out that this argument even yields a quantitative lower bound on the first prime number for which we can claim that $\Gamma_p=\G_p(\F_p)$ in terms of the height of the generating set of $\Gamma$. Namely:

\begin{theorem}(Strong approximation, quantitative version)\label{sappq} Suppose $\G \leq \GL_d$ is a connected, simply connected, semisimple algebraic group defined over $\Q$. Then there are constants $p_0,C_0 \geq 1$ such that if $S \subset \G(\Q)$ is a finite symmetric set generating a Zariski-dense subgroup $\Gamma=\langle S \rangle$ of $\G$, and $M_S$ denotes the maximal height of an element of $S$, then for every prime number $p> \max\{p_0,M_S^{C_0}\}$, the reduction $\Gamma_p$ of $\Gamma$ is equal to $\G_p(\F_p)$.
\end{theorem}

Here the height $H(s)$ of an element $s \in \GL_d(\Q)$ is defined naively as the maximum of the numerators and denominators appearing in the expressions of the matrix coefficients of $s$ as irreducible fractions. The bound $p_0$ is related to the bound $c(n)$ from Nori's theorem and to $p_M$ from Lemma \ref{densep} below. There is very little control on this bound in general (see \cite[Appendix]{salehi-golsefidy-varju} for a discussion of this issue).

Several other proofs and extensions of Theorem \ref{strongapp} (to groups defined over number fields,  to positive characteristic etc.) have since been found. For those we refer the reader to the original articles, in particular \cite{weisfeiler}, \cite{nori}, \cite{hrushovski-pillay}, \cite{pink}, and to the chapter on strong approximation in the recent book by Lubotzky and Segal \cite{lubotzky-segal} or in Nikolov's lecture notes in \cite[chapter II]{klopsch-nikolov-voll}. We also recommend reading Rapinchuk's recent survey \cite{rapinchuk-survey}, which gives a thorough overview of strong approximation.



\bigskip

We now pass to the derivation of Theorem \ref{sappq} from Nori's theorem. First, we replace the naive height with another height, which is better suited for our purposes since it is sub-additive. Given $a \in \GL_d(\Q)$, set
$$h(a):=\sum_{p,\infty} \log^+{||a||_p},$$
where the sum is over all prime numbers $p$ as well as the infinite place $\infty$. Here $\log^+:=\max\{\log,0\}$, and $||a||_p$ denotes $\max_{ij} |a_{ij}|_p$, the maximum $p$-adic absolute value of a matrix entry $a_{ij}$ of $a$, while $||a||_\infty$ is the operator norm of $a$ for the standard euclidean norm on $\R^d$. The following is straightforward:

\begin{lemma} (a) The height $h(a)$ is sub-additive, i.e.  $\forall a,b \in \GL_d(\Q)$,
$$h(ab) \leq h(a) + h(b),$$
 and
(b) it is comparable to the naive height $H(a)$, namely $\forall a$,
$$H(a) \leq e^{h(a)} \leq d (H(a))^{d^2}.$$

\end{lemma}

We conclude that for all $a_1,\ldots,a_n \in \GL_d(\Q)$,

\begin{equation}\label{submul}
H(a_1\cdot \ldots \cdot a_n) \leq d^n (H(a_1) \cdot \ldots \cdot H(a_n))^{d^2}
\end{equation}

Combined with the next lemma, this inequality allows us to assume, in the proof of Theorem \ref{sappq} that $\Gamma$ is generated by two elements, i.e. that  $S:=\{1,a^{\pm 1},b^{\pm 1}\}$.

\begin{lemma}(Reduction to $2$ generators) Let $\G$ be a semisimple algebraic group over $\C$. Then there is $c>0$ such that given any finite symmetric subset $S \subset \G(\C)$, with $1 \in S$, generating a Zariski dense subgroup of $\G$, the bounded power $S^c$ contains two elements $a,b$ which alone already generate a Zariski-dense subgroup.
\end{lemma}

\begin{proof} This is Proposition 1.8. from \cite{breuillard-annals}. The proof is fairly classical, and relies on Jordan's theorem and the Eskin-Mozes-Oh escape from subvarieties lemma (see e.g. \cite[Lemma 3.11]{breuillard-green-tao-linear}).
\end{proof}

\begin{lemma}(Generating is an algebraic condition)\label{dense} Let $\G\leq \GL_d$ be a semisimple algebraic group defined over $\Q$.
There is a proper closed algebraic subvariety $\X \leq \G \times \G$ defined over $\Q$, whose points are precisely the pairs of elements in $\G$ which are contained in a proper algebraic subgroup of $\G$.
\end{lemma}


\begin{proof} This is well-known (see e.g. \cite[Theorem 11.6]{guralnick-tiep}).  We work over an algebraic closure of $\Q$ and  show that $\X$ is a closed algebraic subset. Since $\X$ is invariant under Galois automorphisms, it will automatically be defined over $\Q$. We claim that there are finitely many absolutely irreducible finite dimensional non trivial modules of $\G$, say $\rho_1,\ldots,\rho_k$ such that a subgroup $\Gamma \leq \G$ is not Zariski-dense if and only if $\rho_i(\Gamma)$ fixes a line in the representation space $V_i$ of $\rho_i$ for some $i=1,\ldots,k$. And this happens if and only if $\rho_i(\Gamma)$ fixes a non trivial subspace of $V_i$ for some $i=1,\ldots,k$. This last condition clearly forms an algebraic condition, because it is equivalent to say that $\rho_i(\Gamma)$ does not span the ring of endomorphisms of $V_i$. Moreover the span of $\rho_i(\Gamma)$ is spanned by the $\rho_i(w(a,b))$'s for a bounded set of words $w$. So we indeed have an algebraic condition on the pair $a,b$. Finally $\X$ is proper, because every semisimple algebraic group can be generated by two elements (see e.g. \cite{kuranishi}).

To prove the claim, note that if $\HH$ is a proper closed algebraic subgroup of $\G$, then either it is finite in projection to one of the simple factors of $\G$, or its Lie algebra is not preserved under the adjoint action of $\G$ on $Lie(\G)$. Let $j(d)$ the bound from Jordan's theorem, so that every finite subgroup of $\GL_d$ has a normal abelian subgroup of index at most $j(d)$. For each simple factor $\G_i$ pick an irreducible module whose dimension is larger than $j(d)$, so that no finite subgroup of $\G_i$ can act irreducibly on it. We thus have found finitely many irreducible modules, say $\pi_1,\ldots,\pi_m$ of $\G$ with the property that if a subgroup acts irreducibly on each of them, it must be Zariski-dense. Adding to this list all the non trivial irreducible submodules of the wedge powers $\Lambda^* \pi_i$, we obtain the desired list of modules $\rho_1,\ldots,\rho_k$.
\end{proof}

Now, reducing modulo a large prime $p$, we obtain:

\begin{lemma}(Generating mod p)\label{densep} With the assumptions of the previous lemma, there is $M_0 \geq 1$ such that $\forall M\geq M_0$, there is $p_M>0$ such that if $p>p_M$ is a prime number, the reduction of $\X$  mod $p$ is a proper algebraic subvariety of $\X_p \leq \G_p \times \G_p$ defined over $\F_p$  whose points are precisely the pairs of elements in $\G_p$ which are contained in a proper algebraic subgroup of $\G_p$ of complexity at most $M$.
\end{lemma}

\begin{proof} First observe that there is a bound $M_0$ such that every proper algebraic subgroup of $\G$ is contained in a proper algebraic subgroup of complexity at most $M_0$. This follows from the discussion in the proof of Lemma \ref{dense}, since a proper algebraic subgroup will either stabilize a subalgebra of $Lie(\G)$ which is not an ideal, or will stabilize a proper subspace of some $V_i$. Each of these stabilizers have bounded complexity. Now to prove  the lemma we argue by contradiction. If no such $p_M$ can be found, there must be an infinite sequence of primes $p_i< p_{i+1}$ and pairs $(a_i,b_i) \in \G_{p_i}(\overline{\F_{p_i}})$ such that either for all $i$, $(a_i,b_i) \in \X_{p_i}$ and are not contained in a proper algebraic subgroup of $\G_{p_i}$ of complexity at most $M$, or for all $i$, $(a_i,b_i) \notin \X_{p_i}$ and are contained in a proper algebraic subgroup of $\G_{p_i}$ of complexity at most $M$. The ultraproduct of the $\X_{p_i}$ coincides with $X \otimes_\Q K$, where $K$ is the ultraproduct of the finite fields $\F_{p_i}$. This gives rise to a pair $(a,b)$ in the associated ultraproduct, which, in the first case, belongs to $\X(K)$ and generates a Zariski-dense subgroup, and in the second case does not belong to $\X(K)$ and yet generates a subgroup contained in a proper algebraic subgroup of complexity at most $M$. In both cases we have a contradiction with the definition of $\X$ in Lemma \ref{dense}. For more details on similar ultraproduct arguments, we refer the reader to the appendix of \cite{breuillard-green-tao-linear}.
\end{proof}

Now comes the point where Nori's theorem is used in the form of Corollary \ref{suffdense} : when $\G_p$ is simply connected every subgroup of $\G_p(\F_p)$ which is not contained in an algebraic subgroup of bounded complexity must be all of $\G_p(\F_p)$.

We may then complete the proof of Theorem \ref{sappq}. Pick polynomial functions $(P_k)_{k=1,\ldots,k_0}$, $P_k=P_k((a_{ij},b_{ij}))$, in pairs of matrices $(a,b)$ in $\GL_d$,  which generate the radical ideal of polynomial functions vanishing on $\X$ in $\G \times \G$. We may assume that the $P_k$'s have integer coefficients. If $S=\{1,a^{\pm 1}, b^{\pm 1}\} \subset \G(\Q)$ generates a Zariski-dense subgroup of $\G$, then $(a,b) \notin \X$ and there must exist $k$ such that $P_k(a,b) \neq 0$. We may bound the height of $P_k(a,b)$ in terms of the heights of $a$ and $b$ and the heights of the coefficients of $P_k$. Hence

$$H(P_k(a,b)) \leq O(H(a)H(b))^{O(1)} \leq  (2M_S)^{C},$$
for some constant $C$ depending only on $\G$ and not on $k,a,b$, where $M_S=\max\{H(a),H(b)\}$. This means that if $p>(2M_S)^C$, then $P_k(a,b)$ does not vanish modulo $p$. Now  Lemma \ref{dense}, combined with Nori's theorem (in the form of Corollary \ref{suffdense}), tells us that if additionnally $p$ is larger than a constant depending on $\G$ only, then the reduction mod $p$ of the pair $(a,b)$ generates all of $\G_p(\F_p)$ and we are done. This ends the proof of Theorem \ref{sappq}.

\section{The Bourgain-Gamburd expansion machine}\label{bg-sec}

Bourgain and Gamburd, in their groundbreaking paper \cite{bourgain-gamburd}, came up with a new method to establish the expander property for Cayley graphs of finite groups. They applied it to prove Theorem \ref{ssapp} in the special case of subgroups of $SL_2(\Z)$, but their method is very general. We call it the \emph{Bourgain-Gamburd expansion machine}.
In this section we give an overview of this machine, suitable for the proof of Theorem \ref{ssapp} in full generality.

Let $G_0$ be a finite group, and $S_0 = \{s_1,\dots,s_k\}$ be a symmetric generating set for $G_0$. As before we write:
\[ \mu = \mu_{S_0} := \frac{1}{k}(\delta_{s_1} + \dots + \delta_{s_k})\]
for the uniform probability measure on the set $S$, where $\delta_x$ is the Dirac mass at $x$. For us a probability measure on $G_0$ is the same thing as a function on $G_0$ taking non-negative values at each element of $G_0$ and summing to $1$.

We write
\[ \mu^{n} := \mu \ast \dots \ast \mu\]
for the $n$-fold convolution power of $\mu$ with itself, where the convolution $\mu_1 \ast \mu_2$ of two functions $\mu_1, \mu_2 \colon G_0 \to \R^+$ is given by the formula
\begin{equation}\label{convdef}
\mu_1 \ast \mu_2(g) := \sum_{x \in G_0} \mu_1(g x^{-1}) \mu_2(x).
\end{equation}
The function $x  \mapsto \mu^{n}(x)$ is a probability measure describing the distribution of a random walk of length $n$ starting at the identity in $G_0$ and with generators from $S$. In particular, if $A$ is a subset of $G_0$,
\begin{equation}\label{munh}
\mu^{n}(A) = \P_{w \in W_{n,k}}( w(a_1,\ldots,a_k) \in A ),
\end{equation}
where $W_{n,k}$ is the space of all formal words (not necessarily reduced) on $k$ generators of length exactly $n$. We can now state a version of the Bourgain-Gamburd machine, adapted
 from \cite{bggt} and \cite{varju}.

\begin{proposition}[Bourgain-Gamburd machine]\label{machine}
Suppose that $G_0$ is a finite group, that $S_0 \subseteq G_0$ is a symmetric generating subset, and that there are constants $0 < \kappa, \beta < 1$ such that the following properties hold for every quotient $G$ of $G_0$.
\begin{enumerate}
\item \textup{(High multiplicity)}. For every faithful representation $\rho \colon G \to \GL_d(\C)$ of $G$, $\dim \rho \geq |G|^{\beta}$;
\item \textup{(Classification of Approximate Subgroups)}. For every $\eps>0$, there is $\delta=\delta(\eps)$, $0 < \delta < \eps$ with the property that every $|G|^{\delta}$-approximate subgroup  $A$ of $G$, is either of size $|A|\geq |G|^{1-\eps}$ or is contained in at most $[G:H]^{\eps}/|G|^\delta$ left cosets of a subgroup $H\leq G$;
\item \textup{(Non-concentration estimate)}. Let $S$ be the image of $S_0$ in $G$. There is some even number $n \leq  \log |G|$ such that for all subgroups $H \leq G$,
\[ \mu_S^{n}(H) \leq [G:H]^{-\kappa}.\]
\end{enumerate}
Then the first non zero eigenvalue of the Cayley graph $\Cay(G_0,S_0)$ satisfies
$$\lambda_1 \geq \beta \cdot e^{-\frac{C}{\delta}},$$
where $\delta:=\delta(\eps)>0$ with $\eps:=\min\{\beta,\kappa\}/4$ and $C$ is an absolute constant.
\end{proposition}

We will discuss approximate subgroups in the next section. It suffices for now to say that by definition, given a parameter $K\geq 1$, a \emph{$K$-approximate subgroup} of $G_0$ is a finite symmetric set $A$ containing $1$ such that $AA \subset XA$ for some subset $X \subset G_0$ of size at most $K$.



\begin{remark} We already observed that if the Cayley graph $\mathcal{G}(G_0,S_0)$ is an $\eps$-expander, then so are all induced quotient Cayley graphs corresponding to a quotient group $G:=G_0/H$, for any normal subgroup $H \leq G_0$. It is therefore very natural that the Assumptions (i) to (iii) are made on all quotients of $G_0$.
\end{remark}

As mentioned earlier, Assumption (ii), the classification of approximate subgroups of $G_0$, and (iii), the nonconcentration estimate,  really constitute the beefy parts of the proof of the expander property. They will be dealt with in the next sections.  We also remark that (iii) is the only condition of the three that actually involves the set $S$. Finally we stress that the lower bound on $\lambda_1$ obtained here is independent of the size $k$ of $S$.

\bigskip

An interesting feature of (iii) is that, unlike (i) and (ii), it is \emph{necessary} in order to verify the expander property, because the simple random walk on an expander graph will equidistribute in logarithmic time. Indeed we have the following basic lemma (recall the definition of $\eps$-expanders in (\ref{lambdadef}) above).

\begin{lemma}[random walk characterization of expanders]\label{rwchar} Let $G_0$ be a finite group and $S_0$ a symmetric generating subset not contained in a coset of a subgroup of index $2$ of $G_0$.
\begin{itemize}
\item if the Cayley graph $\mathcal{G}(G_0,S_0)$ is an $\eps$-expander, then there is $C=C_\eps>0$ such that for every $n \geq C \log |G_0|$,
\begin{equation}\label{equid}
\max_{x \in G_0} |\mu_{S_0}^n(x) - \frac{1}{|G_0|}| \leq \frac{e^{-n/C}}{|G_0|^{10}}
\end{equation},

\item if $(\ref{equid})$ holds for some $n \leq C\log |G_0|$, and $C>20$, then $\mathcal{G}(G_0,S_0)$ is an $\eps$-expander, with $\eps=\frac{10}{C}$
\end{itemize}
\end{lemma}

\begin{proof}Let $T_\mu=1-\Delta$ be the operator $f \mapsto \mu * f$ on $\ell^2(G_0)$. To prove the second item, pick an eigenfunction $f$ of the Laplacian with eigenvalue $\lambda_1$ and note that $||(T_\mu^n - \frac{1}{|G_0|}Id)f||_2 \leq \frac{1}{|G|^{10}}||f||_2$ forcing $(1-\lambda_1)\leq |G_0|^{-10/n}$. As for the first item, note that the left hand side of $(\ref{equid})$ is bounded by $||T_\mu||^n \leq ||T_\mu||^{C_\eps \log |G_0|} = 1/|G_0|^{-C_\eps \log (1/||T_\mu||)}$. The assumption on $S_0$ and the fact that $\mathcal{G}$ is the Cayley graph of a group ensure that it is not bi-partite and that $||T_\mu||\leq e^{-c_\eps}$, for some $c_\eps>0$ depending only on $\eps$ and $|S_0|$ (see \cite[Prop. E.1]{bggt}). The result follows with $C_\eps=10/c_\eps$.
\end{proof}

To see that $(iii)$ is necessary, simply note that $\mu^{nm}(H) \geq (\mu^n(H))^m$ and apply the first item in the above lemma to evaluate $\mu^{nm}(H)$ using some $m$ between $C_\eps$ and $2C_\eps$ say.

\begin{remark}\label{nosmallindex} According to result of Landazuri-Seitz \cite{landazuri-seitz}, Assumption (i) is always verified when $G_0$ is a simple or quasi-simple group of Lie type of bounded rank, with the parameter $\beta>0$ depending only on the rank. See Prop. \ref{quasi} below. Looking at the action by translation on $\ell^2(G_0/H)$, where $H$ is an arbitrary subgroup of $G_0$, this implies that every proper subgroup of $G_0$ has index at least $|G_0|^c$ for some $c>0$ depending only on the rank of $G_0$.
\end{remark}

We now pass to the proof of Proposition \ref{machine}. The following basic observation relates the eigenvalues of the Laplace operator $\Delta$ on the Cayley graph, with the probability of return to the identity of the simple random walk. Let $1=\alpha_0 > \alpha_1 \geq \ldots \geq  \alpha_{|G_0|-1}$ be the eigenvalues of the convolution operator
$$T_\mu: f \mapsto \mu * f$$ on $\ell^2(G_0)$. Since $T_\mu=T_{\mu_{S_0}}=Id - \Delta$, the first non trivial eigenvalue of $\Delta$, is just $\lambda_1= 1-\alpha_1$.

Now observe that the eigenspace of $T_\mu$ corresponding to the eigenvalue $\alpha_1$ is invariant under $G_0$ and thus forms a linear representation of $G_0$. Up to replacing $G_0$ with its image modulo of the kernel of this representation, and $\mu$ with the corresponding push-forward measure, we may assume that $G_0$ acts faithfully on this eigenspace. And hence, applying Assumption (i), that the dimension of this eigenspace is at least $|G_0|^\beta$.

Thus we seek a lower bound on $1-\alpha_1$. For this, we write the following naive \emph{trace formula}, which consists in expressing the trace of $T_{\mu^n}=T_\mu^n$ in two ways (this key idea is analogous to what is done in the context of discrete groups in Sarnak-Xue \cite{sarnak-xue} and Gamburd \cite{gamburd-thesis}). Firstly:

$$tr(T_{\mu^n}) = \sum_{x \in G_0} \langle (T_\mu)^n \delta_x, \delta_x \rangle = |G_0| \langle (T_\mu)^n \delta_1, \delta_1 \rangle = |G_0| \mu^n(1),$$
where $\mu^n(1)$ is the value at the identity of the probability measure $\mu^n$. Here $\delta_x$ denotes the Dirac mass at $x$ and $\langle \cdot, \cdot, \rangle$ the $\ell^2$ scalar product on $G_0$. And secondly:

$$tr(T_{\mu^n}) = \alpha_0^n + \alpha_1^n + \ldots + \alpha_{|G_0|-1}^n.$$
We will now play the multiplicity lower bound on $\alpha_1$ against the combinatorial upper bound on $\mu^n(1)$. Since $\alpha_1^n$ appears at least $|G_0|^\beta$ times in the above sum,  discarding all other eigenvalues (note that $n$ is even and hence $\alpha_i^n \geq 0$), we get the following:

\bigskip

\noindent {\bf Observation 1.} If $\mu^n(1) \leq \frac{1}{|G_0|^{1-\beta/2}}$ for some even integer $n \leq C_1 \log |G_0|$, then the first non trivial eigenvalue $\alpha_1$ of $T_\mu$ satisfies
$$\alpha_1 \leq e^{-\frac{\beta}{2C_1}}.$$

Assumption (iii) only guarantees the existence of an even integer $n_0 \leq  \log |G_0|$ such that $\mu^{n_0}(1) \leq \frac{1}{|G_0|^{\kappa}}$ for some positive $\kappa$ which may be smaller than $1-\beta/2$. So in order to conclude, we need to show that $\mu^n(1)$ will decay from $1/|G_0|^{\kappa}$ at time $n=n_0 \leq  \log |G_0|$ to $1/|G_0|^{1-\beta/2}$ at a not much larger time $n=n_1 \leq C_1 \log |G_0|$ for some constant $C_1$ depending only on the constants at hand and not on the size of $G_0$.

Before going further, let us record the following simple remarks:

\begin{remark} When $n$ tends to infinity $\mu^n(1)$ converges to $1/|G_0|$, the uniform distribution on $G_0$.
\end{remark}

\begin{remark} Since $\mu$ is assumed symmetric,

\begin{equation}\label{dec}
\mu^{2n}(1) = \sum_{x \in G_0} \mu^n(x) \mu^n(x^{-1}) = ||\mu^n||_2^2
\end{equation}
\end{remark}
\begin{remark}\label{decH} For every subgroup $H \leq G_0$, the sequence $\mu^{2n}(H)$ is non-increasing: indeed $\mu^{2n}(H)=||f_{n,H}||_2^2$, where $f_{n,H}: G_0/H \to \R$, $gH \mapsto \mu^n(gH)$, and $f_{n+1,H}=T_\mu f_{n,H}$, while $T_\mu$ is a contraction in $\ell^2$.

\end{remark}

The key ingredient in proving this final decay of $\mu^n(1)$ from $1/|G_0|^{\kappa}$ to $1/|G_0|^{1-\beta/2}$ is the following $\ell^2$-flattening lemma, due to Bourgain-Gamburd. It says in substance that the only reason why the convolution of a probability measure with itself  would not decay in $\ell^2$-norm is because it gave a lot of mass to (a coset of) an approximate subgroup.

\begin{lemma}($\ell^2$-flattening lemma) \label{flattening}. There is absolute constant $R>0$ such that the following holds. Let $K \geq 2$ and  $\nu : G_0 \rightarrow \R^+$ be a probability measure on a finite group $G_0$ which satisfies
$$\Vert \nu \ast \nu \Vert_{2} \geq \frac{1}{K}\Vert \nu \Vert_{2},$$
where convolution is defined in \eqref{convdef}. Then there is a $K^R$-approximate subgroup $A$ of $G_0$ with
$$ K^{-R} \frac{1}{\Vert \nu \Vert_{2}^2} \leq |A| \leq K^{R} \frac{1}{\Vert \nu \Vert_{2}^2}$$
and such that for each $x \in A$,
$$\nu * \nu^{-1}(x) \geq \frac{1}{K^{R}|A|}.$$
\end{lemma}

Here $||\nu||_2$ denotes the $\ell^2$ norm on $G_0$, i.e. $||\nu||_2^2:=\sum_{x \in G_0} \nu(x)^2$, and $\nu^{-1}$ denotes the symmetric of $\nu$, namely the probability measure $\nu^{-1}(x):=\nu(x^{-1})$. Observe that the last condition implies immediately that there is $g \in G_0$ such that $\nu(Ag) \geq 1/K^R$.

 \bigskip

\begin{proof} The proof of the $\ell^2$-flattening lemma is really the core of the Bourgain-Gamburd machine. It is derived from a powerful combinatorial tool, the Balog-Szemer\'edi-Gowers lemma (see Lemma \ref{balog} below), due in this context to Tao (\cite{tao-noncommutative}, \cite[\S 2.5, 2.7]{tao-vu}), but which originates from the work of Balog-Szemer\'edi \cite{balog-szemeredi} and from Szemer\'edi's celebrated regularity lemma for large graphs. A simple derivation of the above $\ell^2$-flattening lemma, based on Tao's version of the Balog-Szemer\'edi-Gowers lemma, namely Lemma \ref{balog} below, is given by Varj\'u in \cite[Lemma 15]{varju} and we refer the reader to it for the details. He can also consult \cite[Appendix A]{bggt}. The basic idea is to decompose $\nu$ into approximate level sets $\nu = \sum_i 1_{A_i}\nu$, where $A_i=\{x \in G_0; 2^{i-1}||\nu||_2^2 < \nu(x) \leq 2^i||\nu||_2^2\}$ and show that for some suitable pair $A_{i_1},A_{i_2}$ the number of collisions $||1_{A_{i_1}} * 1_{A_{i_2}}||^2_2$ is large enough to be able to apply Lemma \ref{balog}.
 \end{proof}

Applying this lemma to a symmetric measure $\nu$ with $K=|G_0|^{\delta/R}$, we obtain the following direct consequence:

\begin{corollary}\label{corBSG} Let $0< \delta,\eps \leq \frac{1}{4}$ and let $\nu$ be a symmetric probability measure on a finite group $G_0$ such that $|G_0|^{2\eps} \leq 1/||\nu||_2^2 \leq |G_0|^{1-2\eps}$. Then
$$||\nu * \nu||_2 \leq \frac{1}{|G_0|^{\delta/R}}||\nu||_2,$$ unless there is a $|G_0|^{\delta}$-approximate subgroup $A$ of $G_0$ with $|G_0|^\eps \leq |A| \leq |G_0|^{1-\eps}$ such that $\nu(gA) \geq 1/|G_0|^{\delta}$ for some $g \in G_0$.
\end{corollary}

Here $R$ is the absolute constant from Lemma \ref{flattening}. We are going to apply this corollary several times to the convolution powers $\mu^n$ with even $n$ between $ \log|G_0|$ and $C_1 \log |G_0|$. After only a bounded number of applications of the corollary, $\mu^n(1)$ will be at least as small as $1/|G_0|^{1-\beta/2}$ and we will be done by Observation 1 above.

So we set $\eps=\frac{1}{4}\min\{\beta,\kappa\}$, where $0< \beta\leq 1$ is the exponent of quasirandomness given by Assumption (i) from Proposition \ref{machine} and $\kappa>0$ is given by Assumption (iii).  Let $\delta=\delta(\eps)$ be given by Assumption (ii) of Proposition \ref{machine} (the Classification of Approximate Subgroups).

We will now apply the above corollary to any $\nu$ of the form $\nu=\mu^n$ for some even $n \geq \log |G_0|$. Assume that $||\nu||_2^2 \geq 1/|G_0|^{1-\beta/2}$. Then $1/||\nu||_2^2 \leq |G_0|^{1-2\eps}$, and if $||\nu||_2^2 \leq 1/|G_0|^{2\eps}$, we may apply Corollary \ref{corBSG}, which gives

\begin{equation}\label{decayy}||\nu * \nu||_2 \leq \frac{||\nu||_2 }{|G_0|^{\delta/R}},\end{equation}
unless there is a $|G_0|^{\delta}$-approximate group $A$ in $G_0$ with $|A|\leq |G_0|^{1-\eps}$ such that $\nu(gA) \geq 1/|G_0|^{\delta}$ for some $g \in G_0$. By Assumption (ii) of Proposition \ref{machine}, $A$ must be contained in at most $[G:H]^{\eps}/|G|^\delta$ left cosets of a proper subgroup $H$. Hence at least one coset $xH$ of $H$ charges $\nu$ a lot, i.e. $\nu(xH) \geq 1/[G_0:H]^{\eps}$. However $\nu^2(H) \geq \nu(xH)^2$ since $\nu$ is symmetric, and hence,
 \begin{equation}\label{nubound}
\nu^2(H) \geq 1/[G_0:H]^{2\eps}.
\end{equation}
Since $n \mapsto \mu^{2n}(H)$ is non-increasing (see Remark \ref{decH} above), Assumption (iii) of Proposition \ref{machine} implies that $\nu^2(H) \leq 1/[G_0:H]^\kappa$. However $\kappa>2\eps$, so this clearly contradicts $(\ref{nubound})$.

Therefore  $(\ref{decayy})$ always holds as long as $1/|G_0|^{2\eps} \leq ||\nu||_2^2 \leq 1/|G_0|^{1-\beta/2}$. As a consequence, we need to apply $(\ref{decayy})$ at most a bounded number of times starting from $\nu=\mu^{2n_0}$ with $n_0=[\log |G_0|]$ say to reach the desired upper bound. Note that the bound $1/|G_0|^{2\eps} \geq ||\mu^{2n_0}||_2^2$ holds thanks to Remark \ref{decH}, $(\ref{dec})$ and Assumption (iii) applied to $H=\{1\}$, because $\kappa>2\eps$. Now apply successively $T$ times Corollary $\ref{corBSG}$ to get:

$$||(\mu^{2n_0})^{2^T}||_2\leq  \frac{||\mu^{2n_0}||_2}{|G_0|^{T\delta/R}} \leq \frac{1}{|G_0|^{T\delta/R}} \leq \frac{1}{|G_0|^{1-\beta/2}},$$
provided $T\delta/R \geq 1-\beta/2.$

This yields a constant $C_1$ such that $\mu^{2m}(1) \leq 1/|G_0|^{1-\beta/2}$ for some $m \geq C_1 \log |G_0|$, where an upper bound for $C_1$ is
$$C_1 \leq   2^{\frac{1}{\delta} R(1-\beta/2)}.$$

Together with Observation 1, this finishes the proof of Proposition \ref{machine} with a rather explicit spectral gap, $\alpha_1 \leq e^{-\beta/2C_1}$. Working out the above expression yields the following dependence of the gap in terms of the parameters involved:

$$\lambda_1 \geq \beta \cdot e^{-\frac{C}{\delta}},$$
for some absolute constant $C>0$. Recall that $\delta:=\delta(\eps)$ is the function given in Assumption (ii) with $\eps:=\frac{1}{4}\min\{\beta,\kappa\}$.

\section{Approximate subgroups of linear groups}\label{approxsec}

In this section, we give a very brief introduction to approximate subgroups. The first paragraph gives a definition and some general facts, including the relation with small tripling and the Balog-Szemer\'edi-Gowers lemma. Those are needed only to understand the proof of the $\ell^2$-flattening lemma, Lemma \ref{flattening}, stated in the last section.

Next we describe the classification of approximate subgroups of simple algebraic groups required to deal with Assumption (ii) of the Bourgain-Gamburd machine (Prop. \ref{machine} above) and prove Theorem \ref{class} below, a structure theorem (\cite{breuillard-green-tao-linear,pyber-szabo}) for approximate subgroups of linear groups. Its proof is purely algebro-geometric and requires nothing on approximate subgroups besides the definition. For further introductory material on approximate groups see \cite{tao-noncommutative, breuillard-green-tao-survey, breuillard-msri}.

\subsection{General facts about approximate groups} The notion of an approximate subgroup of an ambient group $G$ was introduced by Terry Tao in \cite{tao-noncommutative} in connection with the work of Bourgain-Gamburd \cite{bourgain-gamburd} and the Balog-Szemer\'edi-Gowers theorem alluded to above in the proof of the $\ell^2$-flattening lemma (Lemma \ref{flattening}). Here is a definition:

\begin{definition}(Approximate subgroup) A (finite) subset $A$ of a group $G$ is said to be a $K$-approximate subgroup of $G$ (here $K \geq 1$ is a parameter) if $A$ is symmetric (i.e. $a \in A \Rightarrow a^{-1} \in A$), contains the identity, and if there is a symmetric subset $X \subset G$ of size $|X|\leq K$ such that
$$AA \subset XA.$$
\end{definition}

Although the definition makes sense without the assumption that $A$ is finite, we will always put this assumption throughout these notes whenever we speak of an approximate subgroup.

Note that $AA=(AA)^{-1} \subset AX$, so we always have $AA \subset XA \cap AX$. Clearly if $K=1$ this notion coincides with the requirement that $A$ be a finite subgroup of $G$.

Although Tao was the first to define approximate subgroups in a non-commutative context, their study in $(\Z,+)$, or $(\R,+)$, is an old subject, part of \emph{additive combinatorics} (see \cite{nathanson}, \cite{tao-vu} for modern expositions), culminating with the so-called Freiman-Ruzsa theorem (\cite{freiman}, \cite{ruzsa}),  which gives a structure theorem for approximate subgroups of $\Z$, or more generally (Green-Ruzsa \cite{green-ruzsa}) abelian groups:

\begin{theorem}(Freiman-Ruzsa, Green-Ruzsa)\label{freiman-ru} Let $G$ be an abelian group and $A \subset G$ be a $K$-approximate subgroup of $G$. Then there is a finite subgroup $H \leq G$ and a centered multidimensional progression $P \subset G$ of dimension at most $d(K)$ such that $A$ is contained in at most $C(K)$ translates of the subset $HP$ and $|HP| \leq C(K)|A|$. The constants $d(K)$ and $C(K)$ depend only on $K$ and not on $G$ nor $A$.
\end{theorem}

By definition a centered \emph{multidimensional progression} of dimension at most $d$ is a subset $P \leq G$ of the form $\pi(B)$, where $\pi: \Z^d \to G$ is a group homomorphism and $B$ is a box in $\Z^d$, namely a subset of the form $\prod_{i=1}^d [-N_i,N_i]$, where the $N_i$'s are non-negative integers. It is easy to see that $B$ is a $2^d$-approximate subgroup, indeed $BB$ is the box with sides $[-2N_i,2N_i]$ and thus can be covered by the translates of $B$ centered at each of the $2^d$ corners of the box $B$. Passing to the quotient via $\pi$, we get that $P$ too is a $2^d$-approximate subgroup, and finally that for every finite subgroup $H \leq G$, the so-called \emph{coset-progression} $HP$ is also a $2^d$-approximate subgroup.

For the proof of this theorem, we refer the reader to the book by Tao and Vu \cite{tao-vu} as well as the article \cite{green-ruzsa} and the original references therein.

\bigskip

Two remarks are in order:

 \begin{itemize}
 \item The bounds $d(K)$ and $C(K)$ can be made quantitative, and good estimates on them are useful for applications as we will see below. Conjecturally (Freiman-Ruzsa conjecture), $d(K)=O(\log K)$ while $C(K)=O(K^{O(1)})$. See Sanders \cite{sanders-survey} for the best currently available bounds.
     \item The conclusion is quite special to abelian groups. A very general structure theorem was recently obtained in \cite{breuillard-green-tao-structure} valid for approximate subgroups of arbitrary groups, but it yields no explicit bounds on $C(K)$. As we will see below, when $G$ is a finite simple group of bounded rank, then a polynomial bound can be given on $C(K)$ provided $A$ generates $G$. Obtaining here a polynomial bound is crucial for the applications to the Bourgain-Gamburd expansion machine, i.e. to Assumption (ii) of Prop. \ref{machine}.
\end{itemize}

As follows immediately from their definition, approximate subgroups do not grow much under self multiplication, namely the product set $A^k:=A \cdot \ldots \cdot A$ of $A$ with itself $k$ times has size at most $|X|^{k-1}|A|$. An important observation (due to Tao using related ideas of Ruzsa) is that we have the following converse:

\begin{proposition}(Small tripling)\label{tripling} Let $A$ be a finite subset of a group $G$ such that $|AAA| \leq K|A|$ for some parameter $K \geq 1$. Then $B:=(A \cup A^{-1} \cup \{1\})^2$ is a $c(K)$-approximate subgroup of size $|B|\leq c(K)|A|$, where $c(K)=O(K^{O(1)})$ and the implied constants are absolute. In particular $|A^n| \leq O(K^{O(n)})|A|$.
\end{proposition}

\begin{proof} The proof is elementary. It is a simple application of the Ruzsa inequality and Ruzsa covering lemma. See \cite[Theorem 3.9]{tao-noncommutative} or \cite[Prop 2.2]{breuillard-clermont}.
\end{proof}

We remark that it is necessary to take the $3$-fold power of $A$ in the assumption of this proposition. It is not true if we only assume that $|AA| \leq K|A|$ (take $A=\{x\} \cup H$, where $x \in G$ and $H$ is a large subgroup such that $xHx^{-1} \cap H =\{1\}$). Nevertheless one can still show in this case that $A$ is covered by $O(K^{O(1)})$ left translates of an $O(K^{O(1)})$-approximate subgroup of $G$ of size at most $O(K^{O(1)})|A|$ (see \cite[Theorem 4.6]{tao-noncommutative})

\bigskip

A deeper fact, recorded in the lemma below, is that one can still identify an approximate subgroup ``near'' the finite set $A$ assuming only that $A$ does not grow under self multiplication in the following statistical sense:

$$||1_A * 1_A||_2^2 = |\{(a,b,c,d) \in A \times A \times A \times A ; ab=cd\}| \geq |A|^3/K.$$
The left hand side is called \emph{the multiplicative energy} of the set $A$ with itself and is sometimes denoted by $E(A,A)$. It is the $\ell^2$-norm squared of the convolution product of the indicator function of $A$ in $G$ with itself and is easily seen to be equal to the expression in the middle (number of ``collisions'' $ab=cd$). In other words: this condition means that the probability that $ab=cd$, when $a,b,c$ and $d$ are chosen at random in $A$ is at least $1/K|A|$. Clearly if $A$ is a subgroup, this probability if exactly $1/|A|$. Also easy to see is the remark that if $|AA|\leq K|A|$, then $||1_A * 1_A||_2^2 \geq |A|^3/K$, indeed setting $r(x):=|\{(a,b) \in A \times A ; ab=x\}|$ we have $\sum r(x)^2 = ||1_A *1_A||_2^2$, $\sum r(x) = |A|^2$ and $|\{x, r(x)>0\}|=|AA|$, hence applying Cauchy-Schwarz:

$$|A|^4 = (\sum r(x))^2 \leq |AA|(\sum r(x)^2) \leq K|A| \cdot ||1_A * 1_A||_2^2.$$

\begin{lemma}(Balog-Szemer\'edi-Gowers-Tao lemma)\label{balog} Suppose $A_1,A_2$ are finite subsets of a group $G$ such that $|A_1|\leq K|A_2|$ and $|A_2|\leq K|A_1|$ and assume that
$$||1_{A_1} * 1_{A_2}||_2^2 \geq (|A_1||A_2|)^{3/2}/K,$$
then there is a $O(K^{O(1)})$-approximate subgroup $A \subset G$ of size $O(K^{O(1)})|A_1|$ such that a subset of $A_1$ of size at least $|A_1|/O(K^{O(1)})$ is contained in some left translate of $A$ and similarly a subset of  $A_2$  of size at least $|A_2|/O(K^{O(1)})$ is contained in some right translate of $A$.
\end{lemma}

\begin{proof} We will not give the proof of this important combinatorial result here. Rather we refer the reader to the book by Tao and Vu \cite[\S 2.5, 2.7]{tao-vu} and Tao's paper \cite[Theorem 5.4]{tao-noncommutative}. See also \cite[Corollaries 4.5, 4.6.]{breuillard-clermont} for a somewhat different argument.
\end{proof}

Note that we cannot claim that $A_1$ itself is contained in few translates of $A$, because if the condition $||1_{A_1'} * 1_{A_2'}||_2^2> (|A'_1||A'_2|)^{3/2}/O(K^{O(1)})$ holds for some subsets $A_1',A_2'$ each making a proportion $\geq 1/O(K^{O(1)})$ of $A_1$ and $A_2$ respectively, then  $||1_{A_1} * 1_{A_2}||_2^2 \geq  ||1_{A_1'} * 1_{A_2'}||_2^2 \geq (|A_1||A_2|)^{3/2}/O(K^{O(1)}).$ For example if $A_1=A_2=\{1,\ldots,N\} \cup \{2,2^2, \ldots,2^N\}$, then $||1_{A_1} * 1_{A_1}||_2^2 \geq ||1_{\{1,\ldots,N\}} * 1_{\{1,\ldots,N\}}||_2^2 \geq N^3$, while $A_1$ is not contained in a bounded number of translates of multidimensional arithmetic progression in $\Z$, hence not contained in a bounded number of translates of an approximate subgroup of $\Z$ (using Theorem \ref{freiman-ru}).

\subsection{Classification of approximate subgroups of $\G(\F_q)$}

The main result here is the following:

\begin{theorem}(Classification theorem)\label{class} Let $K,M \geq 2$. Assume that $\G$ is an absolutely simple algebraic group of complexity at most $M$ defined over an algebraically closed field. If $A$ is a finite $K$-approximate subgroup of $\G$ which is $C$-sufficiently Zariski-dense in $\G$, then either $|A|\leq K^C$, or $\langle A \rangle$ is finite and of cardinality at most $K^{C}|A|$. Here $C=C(M)>0$ is a constant depending only on $M$ and $\dim \G$.
\end{theorem}

The rest of this subsection is devoted to the proof of this theorem and some of its corollaries.

\bigskip

\noindent \emph{Remark.} Although this will not be used later on, we may replace $K^C$ in the above theorem by $CK^{3 \dim \G + 3}$, where $C$ depends again on $M$ and $\dim \G$.

\bigskip

Recall that an affine algebraic variety is said to have complexity at most $M$ if it is the vanishing locus of a finite set of polynomials whose sum of their total degree is at most $M$. This notion can be extended to all algebraic varieties (see \cite[Appendix A]{breuillard-green-tao-linear} for background). Recall further that a subset of $\G$ is called $M$-sufficiently Zariski-dense if it is not contained in a proper algebraic subvariety of complexity at most $M$.

\bigskip

This result was obtained by Green, Tao and the author in \cite[Theorem 5.5]{breuillard-green-tao-linear}. The proof of a closely related statement (in fact Corollary \ref{product} below) was derived independently at the same time by Pyber and Szab\'o, see \cite{pyber-szabo} and \cite{pyber-szabo-msri} for their point of view.

\bigskip

Simple and quasi-simple groups of Lie type are of the form $G=\G(\F_q)^\sigma/Z$, where $\G$ is a simply connected absolutely simple algebraic group defined and split over the prime field $\F_p$, $\sigma$ is a Frobenius map, i.e. the composition of a field automorphism and a graph automorphism, and $Z$ is a central subgroup (whose cardinal is bounded in terms of $\dim \G$ only). It is not difficult (for example using the Lang-Weil bounds or the related and easier Schwarz-Zippel estimates) to check that the subgroups $\G(\F_q)^\sigma$ of fixed points of $\sigma$ are $C$-sufficiently Zariski-dense in $\G$ whenever $q$ is larger than a constant depending only on $C$ and $\dim \G$ (see \cite[proposition 5.4]{bggt} for details). Thus a consequence of Theorem \ref{class} is the following:

\begin{corollary}\label{corfinite} Let $G$ be a (non-abelian) finite simple (or quasisimple) group of Lie type and suppose that $A$ is a $K$-approximate subgroup of $G$. Then either $|A| \leq K^C$, or $|A| \geq |G|/K^C$, or $A$ is contained in a proper subgroup of $G$. Here $C>0$ is a constant depending only on the rank of $G$, not on the size of the associated finite field.
\end{corollary}

\begin{proof} By the discussion above, we may assume that $G$ is a sufficiently Zariski-dense subgroup of a simple algebraic group $\G$ of bounded complexity. It only remains to check that if $A$ generates $G$, then there is a bounded $m$ such that $A^m$ is sufficiently Zariski-dense and then apply Theorem \ref{class} to $A^m$. This fact goes back to Eskin-Mozes-Oh \cite[Prop. 3.2]{EMO}. It is a basic tool called since \emph{escape from subvarieties}, which can be proved with explicit bounds using Bezout's theorem. It can also easily be proved (without an explicit bound on $m$) using ultraproducts: if no such $m$ existed we could form the ultraproduct of possible counter-examples, yielding a subset of $\G(K)$, where $K$ is the corresponding ultraproduct of fields, which generates a subgroup which is not Zariski-dense, hence is contained in a proper algebraic subgroup of $\G(K)$. But this means that most (for the ultrafilter) counter-examples are contained in that algebraic subgroup, contradicting the assumption. See \cite[Lemma 3.11]{breuillard-green-tao-linear} for more details regarding this argument.
\end{proof}

Another related statement is the following, sometimes called the \emph{product theorem}, because it guarantees that any generating subset of $G$ grows under products:

\begin{corollary}(Product theorem)\label{product} Let $G$ be a (non-abelian) finite simple (or quasi-simple) group of Lie type and $A \subset G$ an arbitrary generating finite subset, then
$$|AAA| \geq \min\{|A|^{1+\eps},|G|\},$$
where $\eps>0$ is a constant depending only on the rank of $G$, not on the size of the associated finite field.
\end{corollary}

This result was obtained by Pyber and Szab\'o \cite[Theorem 4]{pyber-szabo}. We show below how to derive it from the classification of approximate subgroups, i.e. Theorem \ref{class}.

\begin{proof} Let $K=|A|^{\eps}$ and apply Proposition \ref{tripling} to get a $(2|A|)^{C\eps}$-approximate subgroup $B$ containing $A$, where $C>0$ is an absolute constant. By Corollary \ref{corfinite}, either $|A| \leq K^C$, or $|A|\geq |G|/K^C$. The first case is ruled out if $\eps<1/2C^2$, because that would force $|A|=1$. In the second case $|A| \geq |G|^{1-\delta}$ for $\delta>0$ which can be taken arbitrarily small provided $\eps$ is small enough. Then a general result of Nikolov-Pyber \cite{nikolov-pyber}, based on an observation of Gowers \cite{gowers} using the quasirandomness of $G$ (i.e. Proposition \ref{quasi} below), implies that $AAA=G$. See \cite[Corollary 2.3.]{breuillard-msri} for a detailed proof of this last step using basic representation theory of finite groups.
\end{proof}

Corollary \ref{product} was first proved by Helfgott \cite{helfgott} in the special cases of $\SL_2(\F_p)$, for the prime field $\F_p$ only, using some \emph{ad hoc} matrix computations based on the sum-product phenomenon from additive combinatorics (i.e. the Bourgain-Katz-Tao theorem \cite{bourgain-katz-tao}). Helfgott later settled the case of $\SL_3(\F_p)$ in \cite{helfgottSL3}.  Earlier work of Elekes and Kir\'aly \cite{elekes-kiraly} had dealt with the analogous result for $\SL_2(\R)$. Although these elementary methods fail to extend to the general case, they have the merit of being somewhat more explicit on the $\eps$ (see e.g. \cite{kowalski-explicit}, \cite{button}).

\bigskip

\noindent \emph{Remark.} Our lower bound on $\eps$ is not explicit. However, if one assumes further that the subset $A$ is $C$-sufficiently Zariski-dense in the ambient simple algebraic group $\G$ (i.e. is not contained in any proper algebraic subvariety of degree, or complexity, at most $C$ for some non explicit $C$ depending only on $\G$), then $\eps$ can be taken to be $1/(3 \dim \G + 4)$. See Remark \ref{explicitconstant} below. The constant $C$ itself (and hence the $\eps$ of Corollary \ref{product}) can be made effective (although not really explicit) using effective algebraic geometry bounds as done by Pyber-Szab\'o in \cite{pyber-szabo}. The treatment in \cite{breuillard-green-tao-linear} was not effective, since it used ultrafilters.

\bigskip

We will sketch below the proof of Theorem \ref{product}. The proof is germane to the proof of the Larsen-Pink theorem \cite{larsen-pink} on the classification of finite subgroups of $\G$. Let us first state a version of the Larsen-Pink theorem appropriate to our discussion (see \cite[Theorem 0.5]{larsen-pink} and \cite{hrushovski-wagner}).

\begin{theorem}(Larsen-Pink theorem)\label{larsen-pink-thm} Let $F$ be an algebraically closed field and $\G$ be an absolutely simple simply connected algebraic group of complexity at most $M$ defined and split over the prime field of $F$. If $\Gamma$ is a finite subgroup of $\G$ which is $C$-sufficiently Zariski-dense in $\G$, then the field $F$ has positive characteristic $p$ and $\Gamma$ is a conjugate of the subgroup $\G(\F_q)$ for a finite field $\F_q \leq F$, $q$ a power of $p$. Here $C=C(M)>0$ is a constant depending only on $M$ and $\dim \G$.
\end{theorem}

This theorem is a strict generalization of Nori's Theorem \ref{suffdense} discussed earlier in the case of simple algebraic groups. However the proof by Larsen and Pink is very different from Nori's counting argument sketched in Theorem \ref{suffdense} above. While Nori was building the algebraic subgroup from below taking products of unipotent elements and using crucially that $p$ is large, Larsen and Pink argue differently and cut the group from above so to speak by computing the approximate size of the centralizers in $\Gamma$ of any subset of elements. This allows them to eventually find many unipotent elements (using an argument similar to the original argument of Jordan \cite{jordan, breuillard-jordan}) including a minimal one which will generate the additive subgroup of the finite field $\F_q$ that we are required to build from $\Gamma$ alone.

In order to compute the correct size of centralizers, Larsen and Pink establish first a very general inequality, the Larsen-Pink non-concentration estimate, which gives an a priori upper bound on the intersection of $\Gamma$ with any algebraic subvariety of bounded complexity. Namely:

\begin{proposition}(Larsen-Pink non-concentration estimate \cite[Thm 4.2.]{larsen-pink})\label{larsen-pink-prop} Under the assumptions of Theorem \ref{larsen-pink-thm}, consider a closed algebraic subvariety $\mathcal{V}$ of $\G$ of complexity at most $M$. Then if $\Gamma$ is a finite subgroup of $\G$ which is $C$-sufficiently Zariski-dense in $\G$,
\begin{equation}\label{lp}
|\Gamma \cap \mathcal{V}| \leq C|\Gamma|^{\frac{\dim \mathcal{V}}{\dim \G}},
\end{equation}
where $C=C(M)>0$ is a constant depending only on $M$ and $\dim \G$.
\end{proposition}

Before we say more about the proof of this proposition and its relation to approximate subgroups, let us explain what it entails for centralizers. Define $q_\Gamma$ as the positive real number $|\Gamma|^{1/\dim \G}$. Let $Z_a$ be the centralizer in $\G$ of an element $a \in \Gamma$. The orbit-stabilizer formula tells us that $$| Z_a \cap \Gamma| \cdot | \{\gamma a \gamma^{-1} ; \gamma \in \Gamma\}| = |\Gamma|,$$
so $$q_\Gamma^{\dim \G}=|\Gamma| \leq |Z_a \cap \Gamma| |\mathcal{V}_a \cap \Gamma| \leq  |Z_a \cap \Gamma| \cdot C q_\Gamma^{\dim \mathcal{V}_a},$$  where $\mathcal{V}_a$ is the conjugacy class of $a$ in $\G$, which is a constructible set in $G$, being the image of $\G$ under the map $g \mapsto gag^{-1}$. We applied the Larsen-Pink inequality $(\ref{lp})$ to the Zariski closure of $\mathcal{V}_a$, which also has dimension $\dim \mathcal{V}_a = \dim \G - \dim Z_a$. Now applying $(\ref{lp})$ once again but this time to $Z_a$ we obtain:

\begin{equation}\label{cent}
\frac{1}{C} q_\Gamma^{\dim Z_a} \leq |Z_a \cap \Gamma| \leq C q_\Gamma^{\dim Z_a}.
\end{equation}
The constant $C$ depends only on the complexity of $Z_a$ and the closure of $\mathcal{V}_a$, which are both bounded in terms of $\dim \G$ and the complexity of $\G$ only and are in particular independent of $a$ (see e.g. \cite[Appendix A]{breuillard-green-tao-linear} for general facts on the complexity of algebraic varieties). So we see that the Larsen-Pink inequality $(\ref{lp})$ not only gives an upper bound, but also a lower bound of the same order of magnitude on the size of centralizers.

The proof of Theorem \ref{class} rests on the same key idea. The main step consists in extending the Larsen-Pink inequality $(\ref{lp})$ to the setting of approximate subgroups:

\begin{proposition}(Larsen-Pink for approximate subgroups)\label{lp-prop-app} Let $K,M \geq 2$. Assume that $\G$ is an absolutely simple algebraic group of complexity at most $M$ defined over an algebraically closed field. If $A$ is a finite $K$-approximate subgroup of $\G$ which is $C$-sufficiently Zariski-dense in $\G$, then for every closed algebraic subvariety $\mathcal{V}$ of $\G$ of complexity at most $M$,
\begin{equation}\label{lpapp}
|A \cap \mathcal{V}| \leq C K^C |A|^{\frac{\dim \mathcal{V}}{\dim \G}},
\end{equation}
where $C=C(M)>0$ is a constant depending only on $M$ and $\dim \G$.
\end{proposition}

This is a strict generalization of $(\ref{lp})$, indeed we recover Proposition \ref{larsen-pink-prop} in the special case when $K=1$ (i.e. $A$ is a subgroup). The possibility of an extension to approximate groups of the Larsen-Pink estimate is an idea of Hrushovski, who proved a qualitative version of $(\ref{lpapp})$ in his ground-breaking paper on approximate groups \cite{hrushovski}. The polynomial dependence of the constant (in $CK^C$) is proved in \cite[Thm 4.1]{breuillard-green-tao-linear} using a variation of the argument we are about to present. Helfgott in \cite{helfgottSL3} proved a special case of this inequality when $\mathcal{V}=T$ is a maximal torus.

\bigskip

\begin{proof} We follow the Larsen-Pink strategy for proving Proposition \ref{larsen-pink-prop}, see \cite[Thm 4.2.]{larsen-pink}. Since a bound on complexity implies a bound on the number of irreducible components (see \cite[Appendix A]{breuillard-green-tao-linear}), it is enough to prove $(\ref{lpapp})$ for irreducible varieties. Clearly the estimate $(\ref{lpapp})$ holds when $\mathcal{V}$ has dimension $0$ or dimension $\dim \G$, so we may pick a possible counter-example to $(\ref{lpapp})$ of minimal positive dimension, say $\mathcal{V}^-$ and another one of minimal co-dimension, say $\mathcal{V}^+$. The basic idea of the proof, which relies crucially on the hypothesis that $\G$ is simple, is that we should be able to find $a \in A$ such that the product $\mathcal{W}:=\mathcal{V}^- a \mathcal{V}^{+} a^{-1}$ is a constructible set of dimension $> \dim \mathcal{V}^{+}$ and thus hopefully will contain too many elements of $AaAa^{-1}$. Hence, since $A$ is an approximate subgroup, some translate of $\mathcal{W}$ will contain too many elements of $A$, contradicting the choice of $\mathcal{V}^+$.

To effect this strategy rigorously, one cannot just proceed as outlined above, because $A \times A$ could concentrate of a singular subvariety of $\mathcal{V}^- \times \mathcal{V}^{+}$ made of non-generic fibers of the product map \begin{eqnarray}\label{maphi}
\Phi: \mathcal{V}^- \times \mathcal{V}^{+} \to \mathcal{W},\\ (x,y) \mapsto xaya^{-1}.
 \end{eqnarray}
 So instead we will prove first a weaker version of $(\ref{lpapp})$ in which  the exponent $\frac{1}{\dim \G}$ is replaced by some $\alpha \in [\frac{1}{\dim \G},1]$. And then improve that estimate by showing that, given any fixed $\beta \geq \frac{1}{\dim \G}$, if the bound $(\ref{lpalph})$ below holds for all subvarieties and for some $\alpha\leq \beta+\eps$, where $\eps=\frac{1}{(\dim \G)^2}$ , then it also holds for $\alpha=\beta$ and all subvarieties:
\begin{equation}\label{lpalph}
|A \cap \mathcal{V}| \leq O(K^{O(1)}) |A|^{\alpha \dim \mathcal{V}}.
\end{equation}
Since $(\ref{lpalph})$ holds obviously when $\alpha=1$ and all subvarieties and since if $(\ref{lpalph})$  holds for $\alpha=\alpha_0$, then it holds for all $\alpha\geq \alpha_0$, this will eventually prove that $(\ref{lpalph})$ holds for $\alpha=\frac{1}{\dim \G}$, so that $(\ref{lpapp})$ holds as desired.

Let us proceed as announced. We fix $\beta \geq \frac{1}{\dim \G}$ and assume that $(\ref{lpalph})$ holds for all $\alpha\geq \beta+\eps$. Pick irreducible subvarieties $\mathcal{V}^-$ and $\mathcal{V}^+$ as above of minimal and maximal dimension providing counter-examples to $(\ref{lpalph})$ for $\alpha=\beta$. This means that $|A \cap \mathcal{V}^+|$ is much bigger than $CK^C |A|^{\beta \dim \mathcal{V}^+}$ and similarly for $\mathcal{V}^-$. By Lemma \ref{transv} below, we may find $a \in A$ such that $\mathcal{W}:=\mathcal{V}^- a \mathcal{V}^{+} a^{-1}$ is a constructible set of dimension $> \dim \mathcal{V}^{+}$. Consider the product map $\Phi$ defined in $(\ref{maphi})$ above. Let $\mathcal{S} \leq \mathcal{V}^- \times \mathcal{V}^+$ be a singular subvariety of strictly smaller dimension outside of which each point lies on a fiber of the right dimension namely $d:= \dim \mathcal{V}^-  + \dim  \mathcal{V}^{+} - \dim \mathcal{W}$. By assumption $d < \dim \mathcal{V}^-$. Basic algebraic geometry (cf. \cite[I.6.3]{shafarevich}) tells us that $\mathcal{S}$ and the fibers are closed algebraic subvarieties, and it is possible to prove by abstract nonsense (see \cite[Appendix A]{breuillard-green-tao-linear}) that their complexity is bounded in terms of those of $\mathcal{V}^{ \pm}$ alone.

Then we see that $A \times A$ must concentrate on $\mathcal{S}$, i.e. $|(A \times A) \cap \mathcal{S}| > \frac{1}{2}|(A \times A) \cap (\mathcal{V}^- \times \mathcal{V}^+)|$, since otherwise, decomposing $(\mathcal{V}^- \times \mathcal{V}^+) \setminus \mathcal{S}$ into fibers of $\Phi$ we would get:

\begin{align*}
CK^C |A|^{\beta(\dim \mathcal{V}^+ + \dim \mathcal{V}^-)} &\ll \frac{1}{2}|A \cap \mathcal{V}^-| \cdot |A \cap \mathcal{V}^+|\\ &\leq  \sum_{z \in \mathcal{W} \cap \Phi(A \times A)} |\Phi^{-1}(z) \cap (A \times A)|\\ &\ll |\mathcal{W} \cap A^4| |A|^{\beta d}
\end{align*}
implying that some translate of $\mathcal{W}$ intersects $A$ in a subset of size much larger than $O(K^{O(1)}) |A|^{\beta \dim \mathcal{W}}$ and thus contradicting the choice (maximality) of $\mathcal{V}^+$. It was licit to bound $|\Phi^{-1}(z) \cap (A \times A)|$ as we did above because of the minimality of $\dim \mathcal{V}^-$ and the fact that $d=\dim \Phi^{-1}(z) < \dim \mathcal{V}^-$.

So we are reduced to the case when $A \times A$ concentrates on the singular subvariety $\mathcal{S}$, which is of dimension at most $\dim \mathcal{V}^- +  \dim \mathcal{V}^+ - 1$. Passing to a proper subvariety of smaller dimension if necessary, we may assume that $\mathcal{S}$ is a subvariety of smallest possible dimension on which $A \times A$ concentrates (i.e. $ |A \cap \mathcal{V}^-| \cdot |A \cap \mathcal{V}^+| \ll O(1) |(A \times A) \cap \mathcal{S}|$).  If its projection to the second factor $\mathcal{V}^+$ is contained in a proper closed subvariety of $\mathcal{V}^+$, then we use $(\ref{lpalph})$ for $\beta+\eps$, to write
$$|(A \times A) \cap \mathcal{S}| \leq O(K^{O(1)}) |A \cap \mathcal{V}^-| \cdot |A|^{(\beta+\eps) (\dim \mathcal{V}^+-1)},$$
which is a contradiction since $(\beta+\eps)(\dim \mathcal{V}^+ -1) < \beta \dim \mathcal{V}^+$. So we may assume that the projection of $\mathcal{S}\leq \mathcal{V}^- \times \mathcal{V}^+$ to the second factor $\mathcal{V}^+$ contains an open dense set of $\mathcal{V}^+$, i.e. the projection is dominant, and hence away from a proper closed subvariety $\mathcal{S}_0$ of $\mathcal{S}$ (on which $A \times A$ cannot concentrate by minimality of $\mathcal{S}$) the fibers of this projection have dimension at most $\dim \mathcal{V}^- -1$. Hence:

\begin{align*}
|(A \times A) \cap \mathcal{S}| &\leq O(1) |(A \times A) \cap \mathcal{S}\setminus \mathcal{S}_0| \\&\leq O(1) \sum_{a \in A \cap \mathcal{V}^+} |(A \times \{a\}) \cap \mathcal{S}\setminus \mathcal{S}_0| \\ &\leq
O(K^{O(1)}) |A \cap \mathcal{V}^+| \cdot |A|^{\beta ( \dim \mathcal{V}^- -1)},
\end{align*}
which is again contradictory. This establishes that $(\ref{lpalph})$ holds for $\alpha=\beta$ and thus by induction that $(\ref{lpapp})$ holds unconditionally.
\end{proof}

\begin{remark}\label{explicitconstant} A careful analysis of the above argument shows that the exponent of $K$ in $(\ref{lpapp})$ can be taken to be $3 \dim \G$, while the multiplicative constant $C$, depends on the complexity of $\mathcal{V}$, and is less explicit owing to the less explicit nature of our notion of complexity and the way it bounds the number of irreducible components (as proved in \cite[Appendix A]{breuillard-green-tao-linear} using ultraproducts). Similarly the threshold of ``sufficient Zariski-density'' of $A$ is non explicit.
\end{remark}


Let $M\geq 2$ and $\G$ as above an absolutely simple connected algebraic group $\G$ of complexity at most $M$. In the above proof, we made use of the following lemma.

\begin{lemma}(Finding a transverse conjugate)\label{transv} There is $C=C(M)>0$ such that the following holds. If $A$ is a $C$-sufficiently Zariski-dense finite subset of $\G$ of complexity at most $M$, then for any two closed algebraic subvarieties $\mathcal{V}_1,\mathcal{V}_2$ in $\G$ of complexity at most $M$ and positive dimension and co-dimension, there is $a \in A$ such that the constructible set $\mathcal{V}_1a\mathcal{V}_2a^{-1}$ has dimension strictly bigger than $\dim \mathcal{V}_2$ and complexity $O_M(1)$ (i.e. a constant depending on $M$ only).
\end{lemma}

\begin{proof} We may assume both varieties to be irreducible. If no such $a$ can be found, then for every $x_1 \in \mathcal{V}_1$ the closed irreducible subvarieties $x_1a\mathcal{V}_2$ and the closure of $\mathcal{V}_1a\mathcal{V}_2$ have same dimension, hence are equal. This means that $x_1a\mathcal{V}_2a^{-1} = x_1'a\mathcal{V}_2a^{-1}$ for all $x_1,x_1' \in \mathcal{V}_1$. Hence that $x_1^{-1}x_1'$ lies in the stabilizer in $\G$ of the subvariety $a\mathcal{V}_2a^{-1}$, namely  $\mathcal{V}_1^{-1}\mathcal{V}_1$ lies in $a\HH a^{-1}$, where $\HH$ is the closed algebraic subgroup $\{g \in \G; g\mathcal{V}_2=\mathcal{V}_2\}$. Since $\mathcal{V}_2$ is a proper subvariety, $\HH$ is a proper subgroup, and since $\G$ is simple $\cap_{a \in \G} a\HH a^{-1}$ is finite. We claim that, because $A$ is assumed sufficiently Zariski-dense, $\cap_{a \in A} a\HH a^{-1}$ is finite too; this will contradict the assumption that $\mathcal{V}_1$ has positive dimension and prove the lemma.

To see the claim, observe that if $\mathbf{Y} \leq \HH$ is an algebraic subvariety of complexity at most $M'$, then $\{g \in \G, \mathbf{Y} \leq g\HH g^{-1}\}$ is a subvariety of complexity at most $O_{M'}(1)$. 
So if $M'=O_M(1)$, then there will be $a \in A$ outside it. Applying this remark several times to each of the irreducible components $\mathbf{Y}$ of the intersections $\HH \cap a_1 \HH a_1^{-1} \cap \ldots \cap a_i \HH a_i^{-1}$, $i\leq k$, we can build $k=O_M(1)$ elements $a_i \in A$ such that $\cap_{1 \leq i \leq k} a_i \HH a_i^{-1}$ has dimension $0$.
\end{proof}

Having the Larsen-Pink estimate for approximate groups (Proposition \ref{lp-prop-app}) at our disposal, we are ready to prove our main theorem. So we now pass to the proof of Theorem \ref{class}. For this it will be convenient to make the following definition:

\bigskip

\noindent \underline{Definition:} A maximal torus $T$ of $\G$ will be called an \emph{involved torus} if $A^2 \cap T$ contains at least one regular element.

\bigskip

Recall that a maximal torus is a connected closed algebraic subgroup of $\G$ containing only semisimple elements (i.e. elements that are diagonalizable in some hence any embedding of $\G$ in $\GL_d$) and maximal for this property. Maximal tori are all conjugate. We refer to Borel's book \cite{borel} or Humphreys \cite{humphreys} for background on algebraic groups. A semisimple element is called regular if its centralizer has a maximal torus of finite index. Regular semisimple elements form a Zariski open subset of $\G$. In particular, since the approximate group $A$ is assumed to be sufficiently Zariski-dense in Theorem \ref{class}, we see that $A$ contains a regular semisimple element. We also recall that every maximal torus $T$ is of bounded index in its normalizer $N(T)$.

We observe at the outset that the number of involved tori is finite, indeed of size at most $|A^2|$, because a regular semisimple element can be contained in at most one maximal torus (the connected component of its centralizer). As in the Larsen-Pink theorem, we set $q_A:=|A|^{1/\dim \G}$. We need to prove that either $q_A$ is $O(K^{O(1)})$, or $\langle A \rangle$ is finite and  $q_A/q_{\langle A \rangle}$ is $O(K^{O(1)})$.

\bigskip
\noindent {\bf Claim 1.} If $T$ is an involved maximal torus, then
\begin{equation}\label{centA}
1/O(K^{O(1)}) q_A^{\dim T} \leq |T \cap A^2| \leq O(K^{O(1)}) q_A^{\dim T}.
\end{equation}

\bigskip

\noindent \emph{Proof:} The argument is the same as the one used to prove $(\ref{cent})$ above applying the Larsen-Pink inequality to both the centralizer and the conjugacy class, and yields the desired estimate for the centralizer $Z(a_0)$ of a regular semisimple $a_0 \in A^2 \cap T$ instead of $T$. Namely looking at the fibers of the map $A \to A^3 \cap \mathcal{V}_{a_0}$, $a \mapsto aa_0a^{-1}$, where $\mathcal{V}_{a_0}$ is the conjugacy class of $a_0$ in $\G$, we see that
$$|A| \leq |A^2 \cap Z(a_0)| \cdot |A^3 \cap \mathcal{V}_0|,$$
but each factor in the right handside is at most $O(K^{O(1)}) q_A^{\dim Z(a_0)}$ and $O(K^{O(1)}) q_A^{\dim \G - \dim Z(a_0)}$ respectively, so the product is $O(K^{O(1)}) |A|$. We thus obtain $(\ref{centA})$ with $Z(a_0)$ in place of $T$. But $Z(a_0)$ is an algebraic subgroup with bounded complexity and $T$ is its connected component, hence $T$ has bounded index in $Z(a)$. This easily implies that $|A^2 \cap T| \geq |A^2 \cap Z(a_0)|/O(K^{O(1)})$ (indeed $A$ will intersect some translate of $Z(a_0)$ in a set of size $\geq q_A^{\dim T}/ O(K^{O(1)})$, hence also some translate of $T$ in a comparable size). This establishes $(\ref{centA})$.

\bigskip

Claim 1 above is really the beef of the proof: assuming only that $T \cap A^2$ has one regular element, we get that it has at least $q_A^{\dim T}$ regular elements up to a $O(K^{O(1)})$ factor. Indeed, the non-regular elements in $T$ are concentrated on a bounded union of proper algebraic subgroups of bounded complexity (the subtori corresponding to the vanishing of some root: in $\SL_n$ this corresponds to the subgroups of diagonal matrices having at least one double eigenvalue). So applying the Larsen-Pink inequality $(\ref{lpapp})$ to this bounded union $T_{sing}$ of subtori, we see that $|A^2 \cap T_{sing}| \leq  O(K^{O(1)}) q_A^{\dim T -1} $. This means that there are at least $q_A^{\dim T} / O(K^{O(1)})$ elements in $A^2 \cap T$ lying outside of $T_{sing}$.

\bigskip
\noindent {\bf Claim 2.}  Unless $q_A$ is $O(K^{O(1)})$, for every maximal torus $T$ of $\G$, if $T$ is involved in $A$, so is $aTa^{-1}$ for every $a \in A$.

\bigskip

\noindent \emph{Proof:} This follows easily from Claim 1 and the above remark. Note that $|A^2 \cap aTa^{-1}|= |a^{-1}A^2a \cap T|$. However $aA^2a^{-1}$ being contained in $A^4$ must lie in at most $K^3$ left translates of $A$. Hence $A^2$ is contained in at most $K^3$ left translates of $a^{-1}Aa$. This means that one of these translates must intersect $T$ in a set of size at least $q_A^{\dim T} / O(K^{O(1)})$. Hence $|a^{-1}A^2a \cap T| \geq q_A^{\dim T} / O(K^{O(1)})$, which implies by the remark above, that $a^{-1}A^2a$ contains a regular semisimple element of $T$, unless $q_A  \leq O(K^{O(1)})$. This proves the claim.

\bigskip

Obviously this lemma implies that all conjugates $gTg^{-1}$, $g \in \langle A \rangle$, are involved.

\bigskip
\noindent {\bf Claim 3.}  Unless $q_A$ is $O(K^{O(1)})$, $\langle A \rangle$ is finite.

\bigskip

\noindent \emph{Proof:} As remarked earlier, since $A$ is sufficiently Zariski-dense, it must contain a regular semisimple element, so there is at least one involved torus.  Since every regular semisimple element is contained in at most one torus, there are only finitely many involved tori. By Claim 2, unless $q_A=O(K^{O(1)})$, they are permuted under conjugation by $\langle A \rangle$. In particular the Zariski-closure $\HH$ of $\langle A \rangle$ intersects the normalizer $N(T)$, and hence $T$ itself in a subgroup of finite index.  We claim that if $\langle A \rangle$ is infinite, and hence the connected component of the identity $\HH^0$ has positive dimension, then there is a closed connected algebraic subgroup $S \leq T$ of bounded complexity and containing $\HH^0$ such that $\HH \leq N(S)$. This will yield the desired contradiction, because $N(S)$ has then bounded complexity. Starting with $S=T$ observe that if $\HH$ does not normalize $S$, then there is $h \in \HH$ such that $S \cap hSh^{-1}$ has dimension $< \dim S$ and bounded complexity. Hence so does the connected component $S_1$ of $S \cap hSh^{-1}$. Since $\HH^0$ is normalized by $h$, this $S_1$ also contains $\HH^0$. Reiterate with $S:=S_1$. This process ends after at most $\dim T$ steps and the claim follows.
\bigskip

The proof of Theorem \ref{class} now follows in a few lines from Claims 1 and 3 and the Larsen-Pink inequality by counting the number $\mathcal{T}$ of involved tori. Since every regular semisimple element is contained in at most one maximal torus, Claim 1 implies that
$$\mathcal{T} \leq O(K^{O(1)})|A^2|/q_A^{\dim T} \leq O(K^{O(1)}) q_A^{\dim \G - \dim T} $$
On the other hand, the subgroup $\langle A \rangle$ acts by conjugation on the (finite) set of involved tori by Claim 2. So
$$\mathcal{T} \geq |\langle A \rangle|/|\langle A \rangle \cap N(T)| \geq |\langle A \rangle|/ O(1) q_{\langle A \rangle}^{\dim T} = q_{\langle A \rangle}^{\dim \G - \dim T}/O(1),$$ where the second inequality follows from the original Larsen-Pink inequality  (Proposition \ref{larsen-pink-prop}) applied to the sufficiently Zariski-dense subgroup $\langle A \rangle$.
So $q_A/q_{\langle A \rangle} = O(K^{O(1)})$ as desired. Finally note that the Larsen-Pink estimate was used only for subvarieties (tori, conjugacy classes, etc.) whose complexity is bounded in terms of the complexity of $\G$ only. Hence the threshold of sufficient Zariski density required in these applications of $(\ref{lpapp})$ is uniform. This ends the proof of Theorem \ref{class}.

\subsection{Verifying Assumption (ii) of the Bourgain-Gamburd machine}\label{verif}

Suppose $G_0=\G(\F_q)$, where $\G$ is an absolutely simple algebraic group defined over the finite field $\F_q$. Then Corollary \ref{corfinite} proved in the previous subsection implies that Assumption (ii) of the Bourgain-Gamburd machine (i.e. Proposition \ref{machine}) holds for $G_0$ with a function $\delta(\eps)$ given by $\delta(\eps)=\eps\min\{\beta,1/(C+1)\}$, where $C$ is the constant from Corollary \ref{corfinite} (distinguishing the cases $H=1$ and $H\neq 1$ and using Remark \ref{nosmallindex}).

In order to deal with products of a bounded number of quasi-simple groups of Lie type of bounded rank, one needs the following rather straightforward extension of Theorem \ref{product}, based on Goursat's lemma about subgroups of direct products of groups.

\begin{theorem}[Approximate subgroups of semisimple groups]\label{semisimpleproduct} Let $G$ be an (almost direct) product of finite simple (or quasisimple) groups of Lie type and suppose that $A$ a $K$-approximate subgroup of $G$. Then either $|A| \geq |G|/K^C$, or $A$ is contained in at most $K^C$ left cosets of a proper subgroup $H$ of $G$, where $C>0$ is a constant depending only on the rank of $G$.
\end{theorem}

We refer the reader to  \cite[Theorem 8.1.]{bggt} for a detailed proof.

If $\G$ is a semisimple algebraic group defined over $\Q$ its reduction $\G_p$ modulo $p$ is well-defined for all but finitely many primes $p$. When $\G$ is simply connected, then $\G_p(\F_p)$ is an almost direct product of quasi-simple groups of Lie type over $\F_q$, where $q$ is a bounded power of $p$. It then follows from Remark \ref{nosmallindex} that every proper subgroup $H$ of a quotient $G$ of $\G_p(\F_p)$ has index at least $|G|^{\eta}$ in $G$ for some $\eta=\eta(\G)>0$ independent of $p$ and of the quotient $G$. We may then take as above $\delta(\eps)=\min\{\eta,1/(C+1)\} \eps$, where $C\geq 1$ is the constant in the above proposition and apply this proposition to $K=|G|^{\delta}$ to obtain Assumption (ii) of the Bourgain-Gamburd machine (Proposition \ref{machine}).

More generally we can handle a bounded number of quasi-simple factors. Namely if $G$ is a (almost direct) product of at most $r$ quasi-simple groups of Lie type of dimension at most $d$ (so for instance if $G$ is the reduction modulo $q:=p_1\cdot \ldots \cdot p_r$, for some distinct large primes $p_1,\ldots,p_r$ of some Zariski-dense subgroup of $\G(\Q)$), then Assumption (ii) of the Bourgain-Gamburd machine is still satisfied with say  $\delta:=\min\{\eta(d)/2r,1/(2C)\}\eps$. Here $\eta(d)>0$ denotes the constant of quasi-randomness (see Remark \ref{nosmallindex}) such that every proper subgroup of a quasi-simple group $S$ of Lie type of dimension at most $d$ has index at least $|S|^\eta$, $r$ is the number of quasi-simple factors of $G$ and $C$ is the constant from Theorem \ref{semisimpleproduct}.

To verify that these constants indeed work, split $G$ as a product $G_1G_2$, where $G_1$ is the product of the quasi-simple factors of size at most $|G|^{\eps/2r}$. Given a $|G|^\delta$-approximate subgroup $A$ of $G$, apply Theorem \ref{semisimpleproduct} to $\pi_2(A)$, the projection of $A$ to $G_2$. Then either $|\pi_2(A)|\geq |G_2|/|G|^{C\delta}$, in which case $|A| \geq |\pi_2(A)| \geq |G|/(|G_1||G|^{C \delta}) \geq |G|^{1-\eps}$, because $|G_1|\leq |G|^{\eps/2}$ and $C\delta \leq \eps/2$. Or $\pi_2(A)$ is covered by at most $|G|^{C\delta}$ translates of a proper subgroup of $G_2$. However proper subgroups of $G_2$ have index at least $|S|^{\eta}$, where $S$ is a quasi-simple factor of $G_2$, hence have index at least $|G|^{\eta \eps/2r}$. It follows that $A$ itself is covered by at most $|G|^{C \delta} \leq [G:H]^\eps/|G|^\delta$ translates of a proper subgroup $H$ of $G$. We are done.

To summarize the above discussion, we have proved in particular:

\begin{corollary} If $\G$ is a semisimple simply connected algebraic group defined over $\Q$ and $p_1,\ldots,p_r$ distinct large enough primes, then Assumption (ii) of Proposition \ref{machine} holds for $G_0:=\prod_{i=1}^r\G_{p_i}(\F_{p_i})$ with $\delta=\eps/Dr$, for some constant $D>0$ depending only on the dimension of the algebraic group $\G$ and not on the $p_i$'s.
\end{corollary}

\section{Random matrix products}\label{rmtsec}

The theory of random matrix products is a well developed part of probability theory on groups. It aims at understanding the statistical behavior of products of $n$ matrices chosen at random when $n$ tends to infinity. It is customary to restrict attention to the case when the matrices are independent and chosen according to the same probability distribution.

In order to establish the non-concentration estimate in the Bourgain-Gamburd machine (i.e. Assumption (iii) in Proposition \ref{machine}) we will need the following result:

\begin{theorem}(probability of return to a subgroup \cite{breuillard-large-deviations}) \label{deviation}Let $\G$ be a connected semisimple algebraic group over a field $K$ of characteristic zero and $\Gamma \leq \G(K)$ a Zariski-dense subgroup generated by a finite set $S$. Let $\mu$ be a probability measure on $S$ with $\mu(s)>0$ for each $s \in S$. Then there is a positive constant $c>0$ such that for every integer $n \geq 1$,
$$\mu^n(\HH)<e^{-cn},$$
uniformly for every proper closed algebraic subgroup $\HH$ of $\G$.
\end{theorem}

We will not go here into all the details of the proof of Theorem \ref{deviation} and instead refer the reader to \cite{breuillard-large-deviations}. However we will indicate how the theory of random matrix products is used to derive it. Theorem \ref{deviation} is deduced from the following fact proved in  \cite{breuillard-large-deviations}.

\begin{proposition}(probability of fixing a line)\label{cont} Let $K$ be a local field of characteristic zero and $\mu$ a probability measure on $\GL_d(K)$ such that $\max\{\|g\|,\|g^{-1}\|\}^\eps$ is $\mu$-integrable for some $\eps>0$. Assume that the support of $\mu$ generates a subgroup $\Gamma_\mu$ which is not relatively compact in projection to $PGL_d(K)$ and does not preserve any finite union of proper vector subspaces of $K^d$. Then there is $c>0$ such that for every $n \geq 1$ and every line $x \in \P(K^d)$,
$$\mu^n(\{g \in \GL_d(K) ; g(x)=x\})<e^{-cn}.$$
\end{proposition}

The condition that the support of $\mu$ does not preserve any finite union of proper subspaces is usually called \emph{strong irreducibility}. It is equivalent to asking that every subgroup of finite index in $\Gamma_\mu$ acts irreducibly, or that the connected component of the Zariski-closure of $\Gamma_\mu$ acts irreducibly. This condition was introduced by Furstenberg in the 1960's in his study of random matrix products \cite{furstenberg}: he showed that under the conditions of the proposition, if $\mu$ is supported on $SL_d(k)$, then the first Lyapunov exponent of $\mu$ is positive, namely:

$$\lim \frac{1}{n} \int \log ||g|| d\mu^n(g) > 0 .$$

Another key theorem in the theory of random matrix products is the simplicity of the Lyapunov spectrum, due to Guivarc'h and Raugi \cite{guivarch-raugi}. It states that under the assumptions of proposition, if the subgroup $\Gamma_\mu$ is proximal (by definition this means that the semigroup $K\Gamma_\mu$ contains a rank one matrix in its closure in the algebra of $d \times d$ matrices $M_{d}(K)$), then the second Lyapunov exponent is strictly smaller than the first. In other words the random matrix product will almost surely contract almost all of the projective space $\P(k^d)$ into an exponentially small neighborhood of a point. From this the conclusion of proposition \ref{cont} can be easily obtained. However this requires the proximality assumption and this assumption does not always hold. It holds for measures $\mu$ supported on Zariski-dense subgroups of $\SL_d(\R)$ due to work of Goldsheid-Margulis \cite{goldsheid-margulis} and this was used by Bourgain and Gamburd in their work \cite{bourgain-gamburd-SLn}. But it does not hold in general in particular if we replace $\R$ with a $p$-adic field. So one needs to avoid this assumption if one wishes to establish Proposition \ref{cont} in full generality (and this generality is require to get Theorem \ref{deviation}). This is what is done in \cite{breuillard-large-deviations}.

\bigskip

Let us now explain how to derive Theorem \ref{deviation} from Proposition \ref{cont}. First we claim that we may assume that $K$ is a local field and that $\Gamma$ is not relatively compact in $\G(K)$. To see it, first note that we may assume $K$ to be finitely generated over $\Q$, since $K$ can be taken to be generated by the matrix entries of the elements of the finite generating set $S$. Now pick a semisimple element of infinite order in $\Gamma$ (it always exists, because $\Gamma$ is Zariski-dense in $\G$) and let $\lambda$ be one of its eigenvalues of infinite order. Find an absolute value on an algebraic closure of $K$, which is not equal to one on $\lambda$ and consider the associated completion to obtain the desired local field. This argument is standard, details can be found in \cite[Lemma 4.1.]{tits-alternative}.

Note that passing to a finite extension of $K$ if necessary, we may assume that $\G$ is $K$-split, so that each absolutely irreducible module of $\G$ can be defined over $K$. Next, we claim that there are a finite number of absolutely irreducible finite dimensional representations of $\G$, say $\rho_1,\ldots,\rho_k$, each of dimension at least $2$, such that every proper closed algebraic subgroup $\HH$ of $\G$ must stabilize a line in one of these representations. This claim was already verified in the proof of Lemma \ref{dense} above.

Finally, note that we may apply Proposition \ref{cont} to each $\rho_i(\Gamma)$, because $\rho_i(\Gamma)$ acts strongly irreducibly on the representation space of $\rho_i$ and is not relatively compact modulo scalars, because it is non relatively compact and of determinant $1$ since $\G$ is semisimple. Since there are only finitely many $\rho_i$'s to consider, we get the desired uniformity in $\HH$ and Theorem \ref{deviation} is proved.

\section{Proof of the super-strong approximation theorem}\label{pfsapp}

In this section we verify that the ingredients of the expansion machine (i.e. Proposition \ref{machine}) are all met under the assumptions of Theorem \ref{ssapp} and complete the proof of this theorem.

In view of Proposition \ref{machine}, we see that Theorem \ref{ssapp} will follow from Proposition \ref{machine} applied to the groups $G_0:=\G_p(\F_p)$ with generating sets $S_p$, where $S_p$ is the reduction modulo $p$ of the generating set $S$ of the Zariski-dense subgroup $\Gamma \leq \G(\Q)$, provided the three assumptions of Proposition \ref{machine} are fulfilled.  We saw in subsection \ref{verif} that Assumption (ii), the classification of approximate subgroups, is satisfied. Let us now consider Assumption (i).

\begin{proposition}[High multiplicity/Quasirandomness]\label{quasi}  Let $\G$ be a semisimple and simply connected algeraic group defined over $\Q$ and $p$ a large enough prime. Then every non-trivial irreducible representation $\rho: \G_p(\F_p) \to \GL_d(\C)$ of $G=\G_p(\F_p)$ has dimension at least $|G|^{\beta}$, where $\beta > 0$ depends only on the dimension of $\G$.
\end{proposition}

\bigskip

\begin{proof} As observed by Sarnak-Xue \cite{sarnak-xue} and Gamburd \cite{gamburd-thesis}, this goes back to Frobenius in the case of $\SL_2$. In \cite{landazuri-seitz} Landazuri and Seitz proved that all non-trivial irreducible
projective representations of a finite simple group of Lie type have dimension at least $|G|^\beta$
for some $\beta > 0$ depending only on the rank, which implies the analogous
claim for irreducible linear representations of any quasi-simple group. Actually, since we do not need the best possible $\beta$, we can arrive to this conclusion rather quickly if we observe that (see e.g. [47, Theorem
4.1]) with the exception of the Suzuki groups, every quasi-simple finite group of Lie type contains a copy of either $SL_2(\F_q)$ or $PSL_2(\F_q)$. But both the Suzuki case and $PSL_2(\F_q)$,
 can be handled easily (see \cite[Lemma 4.1]{landazuri-seitz}).

Now $\G_p(\F_p)$ is an almost direct product of quasi-simple groups over $\F_q$, with $q=p^f$ and $f$ is bounded in terms of the dimension of $\G$ only. So any non trivial linear representation of $\G_p(\F_p)$ gives rise to a representation of a quasi-simple group of Lie type over $\F_{p^f}$ with $f$ and rank bounded in terms of $\dim \G$ only. The proposition follows.
\end{proof}

\bigskip

\noindent \emph{Remark.} Tim Gowers called a group \emph{quasi-random} if it has the property sought for in this proposition. In such groups large subsets behave in a quasi-random way in the sense that the (non-abelian) non trivial characters of the indicator function of a subset are  always very small (\cite{gowers}). This was used by Gowers to show that product-free sets (i.e.  subsets $A \subset G$ not containing any $x,y,z$ with $xy=z$) in such groups are small.

\bigskip

It now remains to verify Assumption (iii) of the Bourgain-Gamburd machine. This is usually the most difficult step. Here it will follow easily from the combination of the quantitative version of the strong approximation theorem proved in Section \ref{nori-strong} and the large deviation estimates from the theory of random matrix products recalled in the previous section.

We may assume that $\G \leq \GL_d$ and this allows us to define the height $H(\gamma)$ of an element of $\Gamma \leq \G(\Q)$ as in Theorem \ref{sappq}. In follows from $(\ref{submul})$ that for every $n \geq 1$ and every $\gamma \in S^n$,
\begin{equation}\label{heightbou}
H(\gamma) \leq (dM_S)^{nd^2},
\end{equation}
where we recall that $M_S$ is defined as
\begin{equation}\label{mdef}
M_S=\max\{H(s),s \in S\}
\end{equation}
and the height $H(s)$ is the naive height (maximum of the numerator and denominator of each matrix entries written as an irreducible fraction).

Fix $\tau>0$ to be determined below. Let $p_0$ be defined as in Theorem \ref{sappq} and $p>p_0$ be any prime number. Choose an even integer $n$ between $\tau \log p$ and $2\tau \log p$. Now let $H$ be a proper subgroup of $\G_p(\F_p)$,  and $S_{H,n}$ be the subset of all elements in $S^n$ whose reduction modulo $p$ lies in $H$. From $(\ref{heightbou})$ we see that if $\tau < 1/(2C_0d^2\log(dM_S))$, then
$$p> (M_{S_{H,n}})^{C_0},$$
where $C_0$ is the constant arising in Theorem \ref{sappq}. Hence Theorem \ref{sappq} applies to the symmetric set $S_{H,n}$ and we conclude that the subgroup generated by $S_{H,n}$ is not Zariski-dense in $\G$. Let $\HH$ be its Zariski-closure.

From Theorem \ref{deviation} we know that in $\Gamma$ and for all $n \geq 1$,

$$\mu_S^n(\HH) \leq e^{-cn},$$
where $c>0$ is a positive constant independent of the choice of $\HH$. However, the reduction mod $p$ map from $\Gamma$ to $\G_p(\F_p)$ is injective on all elements of height at most $p$, and hence on $S^n$, thanks to our choice of $n$ (of size roughly $\tau \log p$). Therefore
$$\mu_{S_p}^n(H) = \mu_{S}^n(\HH) \leq e^{-cn} \leq 1/p^{\tau c} \leq 1/|\G_p(\F_p)|^{\kappa},$$
where we have set $\kappa=c\tau/2d^2$, because $|\G_p(\F_p)|\leq p^{d^2}$.
In particular we see that the exponent $\kappa$ can be taken of the form $c_1\frac{c}{\log M(S)}$, where $c_1>0$ depends only on $\G$ and $c$ is the constant from Theorem deviation.
This establishes the non-concentration estimate needed in the Bourgain-Gamburd machine (Assumption (iii)) and ends the proof of the super-strong approximation theorem (Theorem \ref{ssapp}).

\begin{remark}(Explicit estimate on the gap) The proposed proof of Theorem \ref{deviation} is non effective (it uses the ergodic theorem in Proposition \ref{cont}) and hence gives no explicit lower bound on $c$. However it is likely that $c$ is in fact independent of the choice of $S$ provided $|S|$ is bounded. In that case the estimate given by Proposition \ref{machine} would give the following lower bound for the spectral gap:
$$\lambda_1 \geq  1/M_S^{O(1)},$$ where $M_S$ (see $(\ref{mdef})$) is the maximal height of an element of  $S$ and the implied constant depends only on $\G$ and the cardinal of $S$. See \cite{kowalski-explicit} for an explicit upper bound on the implied constant in the special case when $S$ belongs to $\SL_2(\Z)$.
\end{remark}

\subsection{Several prime factors}
The case of several (but boundedly many primes) can be handled at little additional cost. Assumptions (i) and (ii) of Prop. \ref{machine} have already been verified in this more general setting (see \S \ref{verif}). Assumption (iii) follows in the same way as before by projecting the proper subgroup $H$ to the largest simple factor where it remains proper. The corresponding bound on $\kappa$ and thus on the $\lambda_1$ will depend on the number of prime factors involved.

Hence we get the following improved version of Theorem \ref{ssapp}.

\begin{theorem}\label{severalp} Suppose $\G$ is a connected and simply connected semisimple algebraic group defined over $\Q$ and let $\Gamma \leq \G(\Q)$ be a Zariski-dense subgroup generated by a finite set $S$. Let also $r \in \N$. Then there is $\eps=\eps(S,r)>0$ such that for all large enough distinct prime numbers $p_1,\ldots,p_r$, the projection of $\Gamma$ in the finite group $G_0:=\prod_{i=1}^r \G_{p_i}(\F_{p_i})$ is surjective and the induced Cayley graph of $G_0$ is an $\eps$-expander.
\end{theorem}

Note that the spectral gap in this result depends on $r$ but not on the choice of the $r$ primes $p_1,\ldots,p_r$. Here again, if $\G$ is not assumed simply connected, then the projection of $\Gamma$ to $G_0$ may not be surjective, but it has bounded index (depending only on $\G$ and $r$) in $G_0$ and the induced Cayley graph of the image remains an $\eps$-expander. One reduces easily to the simply connected case by lifting to $\Gamma$ to the simply connected cover of $\G$ (see e.g. \cite[p.399-418]{lubotzky-segal}).

\begin{remark}(Groups defined over a number field)\label{numberfields} If $\Q$ is replaced by a number field $K$, then a similar result holds, which can be reduced to the case of $\Q$. If one wants to take quotients modulo prime ideals $\mathcal{P}$ of the ring of integers $\mathcal{O}_K$ of $K$, then one needs to be careful that the corresponding reduction may not be surjective on $\G(\mathcal{O}_K/\mathcal{P})$ (e.g. $SL_2(\Z)$ is Zariski-dense in $SL_2$, but maps onto $SL_2(\F_p)$ and not onto  $SL_2(\mathcal{O}_K/\mathcal{P}) \simeq SL_2(\F_{p^f})$ for any prime $\mathcal{P}$ with residual degree $f>1$.)
\end{remark}

To palliate this problem, one needs either to pass to a smaller number field (e.g. the one generated by the traces of the elements of $\Gamma$) or to consider the Zariski-closure of the embedding of $\Gamma$ under the restriction of scalars of $\G$ from $K$ to $\Q$. This Zariski closure will be semisimple and Theorem \ref{severalp} will apply. In case $\G$ is not simply connected, one can lift to the simply connected cover. At any case it will always be the case that if $\Gamma$ is a Zariski-dense subgroup of $\G(K)$ for some number field $K$ and semisimple algebraic $K$-group $\G$, then the quotients of $\Gamma$ modulo prime ideals of $\mathcal{O}_K$ will be expanders. This follows readily, by restriction of scalars, from Theoremn \ref{severalp}.

Bourgain-Gamburd-Sarnak \cite{bourgain-gamburd-sarnak} for $\SL_2$, Varj\'u \cite{varju} for $\SL_d$ and Salehi-Golsefidy-Varj\'u \cite{salehi-golsefidy-varju} in general for $\G$ perfect, went much further by establishing that the spectral gap can be made independent of $r$ (for a given $S$). This however requires to prove Assumption (ii) of the Bourgain-Gamburd machine in this setting, hence to understand approximate subgroups of large products of quasi-simple finite groups of bounded rank. This lies much deeper and requires a delicate multi-scale analysis. They prove:

\begin{theorem}[Salehi-Golsefidy-Varj\'u \cite{salehi-golsefidy-varju}]\label{alirezapeter} Let $q_0 \in \N$ and  $\Gamma=\langle S \rangle$ be a finitely generated subgroup of $\GL_d(\Z[\frac{1}{q_0}])$. Assume that the connected component of the Zariski closure of $\Gamma$ is perfect. Then there is $\eps=\eps(d,S)>0$ such that the Cayley graphs of the quotients $\pi_q(\Gamma)$ induced by the generating set $S$ are $\eps$-expanders uniformly over all square-free integers $q$ co-prime to $q_0$. Here $\pi_q$ is the reduction modulo $q$ defined on rational numbers with denominator co-prime to $q$.
\end{theorem}

To finish, let us quote the following related by-product of the Bourgain-Gamburd method.

\begin{proposition}(\cite[Prop. 8.4]{bggt})\label{nocommonfactor} Let $r \in \N$ and $\eps>0$. Suppose $G=G_1G_2$, where $G_1$ and $G_2$ are products of at most $r$ finite simple \textup{(}or quasisimple\textup{)} groups of Lie type of rank at most $r$. Suppose that no simple factor of $G_1$ is isomorphic to a simple factor of $G_2$. If $x_1=x_1^{(1)}x_1^{(2)},\ldots, x_k=x_k^{(1)}x_k^{(2)}$ are chosen so that $\{x_1^{(1)},\ldots,x_k^{(1)}\}$ and $\{x_1^{(2)},\ldots,x_k^{(2)}\}$ are both $\eps$-expanding generating subsets in $G_1$ and $G_2$ respectively, then $\{x_1,\ldots,x_k\}$ is $\delta$-expanding in $G$ for some $\delta=\delta(\eps,r)>0$.
\end{proposition}

The assumption that no simple factor of $G_1$ be isomorphic to a simple factor of $G_2$ is necessary here, because otherwise $\{x_1,\ldots,x_k\}$ may not generate. However what if we suppose it generates, is the conclusion still true without the assumption that $G_1$ and $G_2$ have no isomorphic factors (e.g. if $G_1=G_2=SL_2(\F_p)$) ? this is an open question.

\section{The group sieve method}

One of the leitmotives of the subject matter in this paper is the ability to study finite simple groups of Lie type as quotients of certain infinite linear groups and thereby to do geometry and analysis  on infinite groups in order to derive properties of finite groups, such as the expander property of their Cayley graph. The purpose of the sieve method is to achieve the converse: to study infinite linear groups from the properties of their finite quotients.

In this concluding section, we describe this method, first by showing a very simple application of Theorem \ref{ssapp} to random matrix product theory, where only one prime is required, and then by describing the \emph{group sieve lemma} of Lubotzky and Meiri and two of its applications to the study of \emph{generic properties} in infinite linear groups.

\subsection{Large deviations for subvarieties}

One of the simplest example showing the power of Theorem \ref{ssapp} is the following theorem. It says that random walks on linear groups do not concentrate much on any algebraic subvariety.

\begin{theorem}(Subvarieties are exponentially small) \label{subvariety} Let $K$ be a field of characteristic zero, $\Gamma \leq \GL_d(K)$ a non virtually solvable finitely generated subgroup and $\mu$ a probability measure whose support $S$ is a finite symmetric generating subset of $\Gamma$. Let $\G$ be the Zariski closure of $\Gamma$, and $R$ its solvable radical. Suppose $\mathcal{V}$ is an algebraic subvariety in $\GL_d$ such that $\dim(R(\mathcal{V} \cap \G))< \dim \G$. Then we have for all $n \geq 1$:
$$\mu^n( \Gamma \cap \mathcal{V}) \leq c_0(\mathcal{V}) \cdot e^{-cn},$$
where $c_0(\mathcal{V})>0$ is a constant depending only on the complexity (i.e. degree) of $\mathcal{V}$, and $c>0$ depends only on $\mu$.
\end{theorem}

Note that we have already shown a special case of this theorem in Theorem \ref{deviation} above. Theorem \ref{deviation} claimed essentially the same result when the subvariety $\mathcal{V}$ is assumed to be an algebraic subgroup. Although a direct approach similar to the proof of the Larsen-Pink inequality (Prop. \ref{lp-prop-app}) might be successful in deriving Theorem \ref{subvariety} from Theorem \ref{deviation}, the sieve method here can be implemented without any effort (modulo standard reductions) and yields Theorem \ref{subvariety} as a direct consequence of the super-strong approximation theorem (Theorem \ref{ssapp}) as we now show. This was already observed (and proved in a special case) in the original work of Bourgain-Gamburd \cite[Corollary 1.1]{bourgain-gamburd-SLn}.

\bigskip

\begin{proof} We first reduce to the case when the Zariski-closure $\G$ of $\Gamma$ is semisimple and defined over $\Q$. Taking the quotient modulo the solvable radical $R$, we may assume that $\G$ is semisimple (with connected component of the identity $\G^0$). Now since $\Gamma$ is finitely generated, we may assume that the field $K$ is finitely generated over $\Q$, hence is a finite algebraic extension of a purely transcendental extension of $\Q$ with a finite transcendence basis. One may then specialize and pick algebraic values for this transcendence basis in such way that the connected component of the Zariski closure of the resulting image group $\Gamma'$, now a subgroup of $\GL_d(\overline{\Q}$), is still $\G^0$ (this follows from Lemma \ref{dense}, see also \cite{larsen-lubotzky} for a related statement). Now taking the restriction of scalars to $\Q$ we have reduced to the case when $K=\Q$ and $\G^0$ is semisimple.

It is enough to prove the result for $n$ even, and hence replacing $S$ with $S^2$ we may assume that $1$ belongs to $S$ (note that the subgroup generated by $S^2$ has finite index in $\Gamma$). Let then $\Gamma_0:=\Gamma \cap \G^0$. It is a subgroup of finite index in $\Gamma$ which is Zariski dense in $\G^0$. Now pick a large prime $p$. For $p$ large enough, we know by the super-strong approximation theorem (Theorem \ref{ssapp}) that $(\Gamma_0)_p$, the reduction mod $p$ of $\Gamma_0$, has bounded index in $\G^0_p(\F_p)$ and that its induced Cayley graph is an $\eps$-expander for some $\eps>0$ independent of $p$. It follows that the reduction mod $p$ of $\Gamma$, itself is a finite group $G_p$ containing $(\Gamma_0)_p$ as a subgroup of bounded index and hence is also an $\eps'$-expander for some $\eps'>0$ independent of $p$ and depending only on $\eps$, $\G$, and the index of $\Gamma_0$ in $\Gamma$. Moreover $S_p$ is not contained in a coset of a proper subgroup of $G_p$, because $1 \in S_p$. By the random walk characterization of expanders (see Lemma \ref{rwchar} above), this means that random walks at any time larger than $C_\eps \log p$ are very well equidistributed in the sense that if $n =[C_{\eps'} \log |G_p|]$ say
$$|\mu_p^n(x) - \frac{1}{|G_p|}| \leq 1/|G_p|^{10}$$
for every $x \in G_p$. In particular
$$\mu^n(\mathcal{V}) \leq \mu_p^n( \mathcal{V} \textnormal{ mod } p) \leq \frac{|\mathcal{V}_p|}{|G_p|} + 1/|G_p|^9,$$
However the assumption on $\mathcal{V}$ implies that $|\mathcal{V}_p|\leq c_0(\mathcal{V})p^{\dim \G -1}$ while $|G_p| = \Omega(p^{\dim \G})$ (see the Schwarz-Zippel lemma in \cite{bggt}). If follows that
$$\mu^n(\mathcal{V}) \leq c_0(\mathcal{V})\cdot O(1/p),$$
with the implied constant depending only on $\G^0$. Now given any large $n$, one needs only pick a prime $p$ such that $n$ is roughly of size $C_\eps \log |G_p|$ and the result follows.
\end{proof}

For another method towards Theorem \ref{subvariety} and related partial results see the work of Aoun \cite{aoun}.

We now pass to a corollary of Theorem \ref{subvariety}. In \cite{aoun-free}, R. Aoun showed a probabilistic version of the Tits alternative: he proved that two independent random walks on a non virtually solvable linear group eventually generate a free subgroup. In other words a generic pair of elements always generates a free subgroup.  Combining Theorem \ref{subvariety} with Lemma \ref{dense} we can now assert that a generic pair of elements generates a Zariski-dense free subgroup, namely:

\begin{corollary}(A generic pair generates a Zariski-dense free subgroup) \label{free} Under the assumptions of Theorem \ref{subvariety} assume further that the Zariski closure of $\Gamma=\langle S \rangle$ is connected semisimple. Let $\mathcal{E}$ be the set of pairs $(a,b)$ in $\Gamma \times \Gamma$ such that the subgroup $\langle a, b \rangle$ is either not free, or not Zariski dense in $\Gamma$.  Then there is $c=c(\mu)>0$
$$\mu^n \times \mu^n( \mathcal{E}) \leq e^{-cn}.$$
\end{corollary}

\begin{proof} Aoun's theorem \cite{aoun-free} tells us that for some $c>0$, $\mu^n \times \mu^n( \mathcal{NF}) \leq e^{-cn}$, where $\mathcal{NF}$ is the set of non-free pairs. Now applying Theorem \ref{subvariety} to the group $\Gamma \times \Gamma$ in $\G \times \G$ the measure $\mu \times \mu$ and subvariety $\mathcal{V}=\mathbf{X}$ from Lemma \ref{dense}, we get the desired result.
\end{proof}

For related results, see Aoun's work \cite{aoun} and Rivin's \cite{rivin-zariski}.

\subsection{The group sieve lemma}
The spectral gap for mod $p$ quotients has been exploited by Rivin \cite{rivin} and Kowalski \cite{kowalski-book} to perform sieving on arithmetic lattice subgroups. Prior to the new results on thin groups such as the super-strong approximation theorem, the spectral gap was known in a variety of cases for mod $p$ or mod $n$ quotients of arithmetic subgroups. Thanks to super-strong approximation (i.e. Theorem \ref{ssapp} or \cite{salehi-golsefidy-varju}), we can now perform this sieving on arbitrary Zariski-dense subgroups (i.e. thin subgroups).

In Theorem \ref{subvariety} we used only one prime number to show our non concentration estimate. The power of the sieve consists in taking advantage of several primes and using as a guiding principle that primes are essentially independent.

Lubotzky and Meiri \cite{lubotzky-meiri} formulated the following elegant lemma, which gives a simple set of conditions to be fulfilled in order to get further genericity results (akin to Theorem \ref{subvariety} above) that may require more than one prime.

\begin{lemma}[Group sieve lemma]\label{group-sieve}
Let $\Gamma=\langle S \rangle$ be a group generated by a finite symmetric set $S$ and $\{N_i\}_{1 \leq i \leq N}$ be a finite sequence of finite index normal subgroups. Set $\pi_i: \Gamma \to \Gamma/N_i$ the projection maps. Let $\mathcal{Z} \subset \Gamma$ be a subset of $\Gamma$ and assume that there are positive constants $D,\eps,\alpha$, with $\alpha \in (0,1)$, such that

\begin{itemize}
    \item $\Cay(\Gamma/(N_i \cap N_j), S \textnormal{ mod } N_i \cap N_j)$ for $i \neq j$ are $\eps$-expanders;
    \item $\Gamma/(N_i \cap N_j) \simeq \Gamma/N_i \times \Gamma/N_j$ for $i \neq j$;
    \item $|\Gamma / N_i| \leq N^D$ for all $i=1,\ldots,N$;
    \item $|\pi(\mathcal{Z})|\leq (1-\alpha)|\pi_i(\Gamma)|$ for all $i=1,\ldots,N$.

\end{itemize}
Then there is a constant $B=B(\eps,D,\alpha)>0$ such that for all $n \geq B \log N$,
$$\mu_S^n(\mathcal{Z})\leq \frac{1}{N}.$$
\end{lemma}

As before we have denoted by $\mu_S$ the uniform probability measure on the finite symmetric generating set $S$. Note that only the last assumption involves the set $\mathcal{Z}$. In applying this lemma, typically the $\pi_i$ will be the reduction maps modulo a prime $p_i$. It is crucial that the constant $B(\eps,D,\alpha)$ depends only on these three parameters and not on $\Gamma$, nor the choice of the sequence $\{N_i\}_i$.

The proof of this lemma is quite short, but before we give it in full,  let us comment on it a little. Let $S_n:=Y_1\cdot\ldots\cdot Y_n$ be the product of $n$ independent random variables $Y_1,\ldots,Y_n$ on $\Gamma$ all distributed according to the same probability distribution $\mu_S$ (the uniform distribution on the generating set $S$). The key feature of an expander graph is that random walks on them become equidistributed very fast. By the first item in the above lemma, the Cayley graph of $\Gamma/(N_i \cap N_j)$ is an $\eps$-expander. Clearly this also implies that the quotients $\Gamma/N_i$ and $\Gamma/N_j$ are $\eps$-expanders. Hence the distributions of $\pi_i(S_n)$ and $\pi_j(S_n)$ are very close to the uniform distribution on $\Gamma/N_i$ and $\Gamma/N_j$ respectively as long as $n \geq C_\eps \log |\pi_i(\Gamma)|$, so in particular if $n \geq C_\eps D \log N$ (thanks to the third item). By the second item the natural injection from $\Gamma/(N_i \cap N_j)$ to $\Gamma/N_i \times \Gamma/N_j$ is surjective. This implies that the joint distribution $(\pi_i(S_n),\pi_j(S_n))$ is also close to the uniform distribution, and hence that $\pi_i(S_n)$ and $\pi_j(S_n)$ are almost independent as random variables.

Suppose for a second that they were actually independent. Then quite obviously, using the fourth item in the last inequality:

$$\P(S_n \in \mathcal{Z}) \leq \P(\pi_i(S_n) \in \pi_i(\Z) \textnormal{  }\forall i \leq e^{n/C_\eps D}) \leq (1-\alpha)^{e^{n/C_\eps D}},$$
where $\P(\Omega)$ denotes the probability of the event $\Omega$. We would thus get a super-exponential decay of the probability of belonging to $\mathcal{Z}$.

Of course joint independence is too much to hope for, but the expander property on $\Gamma/N_i \times \Gamma/N_j$ implies that the $\pi_i(S_n)$ are pairwise almost independent. Now the following classical result from basic probability theory (the second moment method) allows us to take advantage of this pairwise almost independence in order to derive a meaningful upper bound on $\P(S_n \in \mathcal{Z})$.

\bigskip

\begin{lemma}\label{moment} Let $X\geq 0$ be a real random variable with $\E(X^2)<\infty$ and $T\geq 1$ a parameter.
\begin{enumerate}
\item (1st moment method) $\P(X \leq T \cdot\E(X)) \geq 1- 1/T$;
\item (2nd moment method) $\P(X \geq \frac{1}{T}\cdot\E(X)) \geq (1-\frac{1}{T})^2 \frac{\E(X)^2}{\E(X^2)}$.
\end{enumerate}
\end{lemma}

\begin{proof} The first item is an instance of Chebychev's inequality: $$\P(X \geq T \cdot \E(X)) \cdot T \cdot \E(X) \leq \E(X 1_{X \geq T \cdot \E(X)}),$$
while the second follows from Cauchy-Schwarz:
$$(1- \frac{1}{T}) \E(X) \leq \E(X 1_{X \geq \frac{1}{T} \cdot \E(X)}) \leq \E(X^2)^{\frac{1}{2}} \P(X \geq \frac{1}{T} \E(X))^{\frac{1}{2}}$$
\end{proof}

Applying this lemma to the variable $X:=\sum_{i=1}^N 1_{A_i^c}$, ($A_i^c$ being the complement of the event $A_i$), we obtain:

\bigskip

\noindent {\bf Fact (exploiting pairwise almost independence) : } if $\{A_i\}_{1 \leq i \leq N}$ are $N$ events on a probability space, such that for some $\alpha, \delta>0$,
\begin{itemize}
\item  $\P(A_i) \leq 1-\omega$ for each $i=1,\ldots,N$, and

\item $\P(A_i \cap A_j) \leq \P(A_i)\P(A_j)+\delta$ for all $i \neq j$,

\end{itemize}
then

$$\P(\cap_{1 \leq i \leq N} A_i) \leq \frac{1}{\omega^2}(\delta + \frac{3}{N}).$$

\begin{proof} Indeed, $\P(A_i^c)= 1-\P(A_i) \geq \omega$, so $\E(X) \geq \omega N$ and by Lemma \ref{moment}
$$1- \P(\cap_1^N A_i) = \P(X \geq 1) \geq \P(X \geq \frac{1}{\omega N} \cdot \E(X)) \geq (1-\frac{1}{\omega N})^2 \frac{\E(X)^2}{\E(X^2)},$$
while $\E(X^2)= \sum_i \P(A_i^c) + \sum_{i \neq j} \P(A_i^c \cap A_j^c)$ and $\E(X)^2= \sum_i \P(A_i^c)^2 + \sum_{i \neq j} \P(A_i^c) \P(A_j^c)$. Hence using that $\P(A_i^c \cap A_j^c) \leq \P(A_i^c) \P(A_j^c) + \delta$,
$$\E(X^2) - \E(X)^2 \leq \sum_i \P(A_i^c)\P(A_i) + \delta N(N-1) \leq N(1- \omega) + \delta N^2,$$
from which we deduce (using that $\E(X) \geq N\omega$) that
$$1- \P(\cap_1^N A_i) \geq (1-\frac{1}{\omega N})^2 (1 - \frac{N(1-\omega) + \delta N^2}{(\omega N)^2}) \geq 1 - \frac{1}{\omega^2}(\delta + \frac{3}{N})$$ as desired.
\end{proof}

We can now complete the proof of the group sieve lemma (i.e. Lemma \ref{group-sieve}):

\bigskip

\begin{proof} Note that we may assume that $n$ is even, and thus replacing $S$ by $S^2$ if necessary we may assume that $S$ contains $1$. Then by the random walk characterization of $\eps$-expanders (Lemma \ref{rwchar}) we know that the random walk $S_n=Y_1\cdot\ldots \cdot Y_n$ is almost equidistributed in projection to each $\pi_i(\Gamma)$ as long as $n \geq C_{\eps} \log |\Gamma/N_i|$, hence as soon as $n \geq C_\eps D \log N$. In particular for all $x \in \pi_i(\Gamma)$:
\begin{equation}\label{eqqs}
|\P(\pi_i(S_n) = x) - \frac{1}{|\pi_i(\Gamma)|}| \leq \frac{e^{-n/C_\eps}}{|\pi_i(\Gamma)|^{10}}
\end{equation}
and for $i \neq j$, $x \in \pi_i(\Gamma)$ and $y \in \pi_j(\Gamma)$
\begin{equation}\label{eqqs2}
|\P((\pi_{i}(S_n),\pi_j(S_n)) = (x,y)) - \frac{1}{|\pi_i(\Gamma)| \cdot |\pi_j(\Gamma)|}| \leq \frac{e^{-n/C_\eps}}{|\pi_i(\Gamma)|^{10}|\pi_j(\Gamma)|^{10}}
\end{equation}

Let $A_i$ be the event ``$\pi_i(S_n) \in \pi_i(\mathcal{Z})$''. From $(\ref{eqqs})$ and $(\ref{eqqs2})$ we get for $i \neq j$
$$|\P(A_i) - \frac{|\pi_i(\mathcal{Z})|}{|\pi_i(\Gamma)|}| \leq \frac{e^{-n/C_\eps}}{|\pi_i(\Gamma)|^{9}},$$
$$|\P(A_j) - \frac{|\pi_j(\mathcal{Z})|}{|\pi_j(\Gamma)|}| \leq \frac{e^{-n/C_\eps}}{|\pi_j(\Gamma)|^{9}},$$
$$ |\P(A_i \cap A_j) - \frac{|\pi_i(\mathcal{Z})|}{|\pi_i(\Gamma)|}\frac{|\pi_j(\mathcal{Z})|}{|\pi_j(\Gamma)|}| \leq \frac{e^{-n/C_\eps}}{|\pi_i(\Gamma)|^{9}|\pi_j(\Gamma)|^{9}},$$
Hence
$$|\P(A_i \cap A_j) - \P(A_i)\P(A_j)|\leq 3e^{-n/C_\eps}$$
Recall further that by assumption $|\pi_i(\mathcal{Z})|/|\pi_i(\Gamma)| \leq 1-\alpha$ hence
$$\P(A_i) \leq 1- \alpha + e^{-n/C_\eps} \leq 1- \alpha/2,$$
for $n$ large enough. Setting $B(\eps,D,\alpha) = 10C_\eps D/\alpha^2$ (say), the lemma now follows by applying the Fact above with $\omega:=\alpha/2$, $\delta=3e^{-n/C_\eps}$.
\end{proof}

In the next subsection, we give an application of this group sieve lemma to a counting problem in infinite linear groups.

\bigskip

To conclude we note that the pairwise almost independence given by the assumption that the Cayley graphs of $\Gamma/(N_i \cap N_j) \simeq \Gamma/N_i \times \Gamma/N_j$ are expanders corresponds to the super-strong approximation theorem for products of two prime factors (i.e. when $r=2$ in Theorem \ref{severalp}). The result of Salehi-Golsefidy and Varj\'u \cite{salehi-golsefidy-varju} shows uniform expansion for an arbitrary (growing) number of prime factors. This corresponds to joint almost independence of the sequence $\pi_i(S_n)$ instead of pairwise. Clearly this is a much stronger property to have at one's disposal and it is crucial in the Affine Sieve of Bourgain-Gamburd-Sarnak \cite{bourgain-gamburd-sarnak} and Salehi-Golsefidy-Sarnak \cite{salehi-sarnak}.


\subsection{Proper powers in linear groups are scarce}

In \cite{lubotzky-meiri} Lubotzky and Meiri use the group sieve lemma (Lemma \ref{group-sieve}) above to establish the following result:

\begin{theorem}(Proper powers are exponentially small, \cite{lubotzky-meiri}) \label{properpowers} Under the assumptions of Theorem \ref{subvariety}, let $\mathcal{P}$ be the proper powers in $\Gamma$, i.e. the set of elements in $\gamma \in \Gamma$ such that there is $\gamma_0 \in \Gamma$ and $k \geq 2$ such that $\gamma=\gamma_0^k$.  Then $\mathcal{P}$ is exponentially small, namely there is $c>0$ such that for all $n \geq 1$,
$$\mu^n( \mathcal{P}) \leq e^{-cn}.$$
\end{theorem}

An old result of Malcev (see \cite{lennox} and references therein) says that for each $n \geq 1$, then set of $n$-th powers in any finitely generated nilpotent group contains a finite index subgroup, and thus cannot be exponentially small. So Theorem \ref{properpowers} can be seen as a strong quantitative converse to Malcev's theorem.  Prior attempts to prove this result, see Hrushovski-Kropholler-Lubotzky-Shalev \cite{hrushovski-kropholler-lubotzky}, could only go as far as proving that for each $k$, the set of $k$-powers in $\Gamma$ does not contain a finite index subgroup of $\Gamma$.

\bigskip

We sketch the proof in the special case when $\Gamma$ is a Zariski-dense subgroup of $SL_d(\Z)$.

\bigskip

\begin{proof} We want to apply the group sieve lemma to the subset $\mathcal{Z}:=\mathcal{P}$ of proper powers. The projection maps $\pi_i$ will be the reduction maps modulo large primes $p_i$ to be chosen carefully. By the strong approximation theorem (Theorem \ref{strongapp} above) $\Gamma$ maps onto $\SL_d(\F_p)$ for all large enough prime $p$.

In a finite group every element of order at least $3$ is a proper power, so we have to restrict attention to $m$-powers (i.e. elements in the image of the map $g \to g^m$) for each given $m$. Luckily we do not need to consider all $m$'s, but only those with $m  \leq C n$ for some $C=C(S)>0$. The reason is that if $\gamma \in SL_d(\Z)$ has an eigenvalue $\lambda$ of modulus $>1$, then it is of modulus $>1+\delta$ for some $\delta$ depending only on the dimension $d$ (indeed eigenvalues are roots of the characteristic polynomial, which has degree $d$ and integer coefficients: if all eigenvalues were say $\leq 2$ in modulus, then the coefficients would be bounded, leaving only finitely many possibilities for $\lambda$). So for every $m\geq 2$, $$||\gamma^m|| \geq |\lambda|^m \geq (1+\delta)^m,$$ while every element in the support of the measure $\mu^n$ has size at most $M_S^n$, where $M_S=\max\{||s||, s \in S\}$. So if an element $g$ in the support of $\mu^n$ is a proper power $\gamma_0^m$, then either $m=O(n)$ or $g$ has all its eigenvalues of modulus $1$. Kronecker's lemma tells us that if the roots of a monic polynomial of degree $d$ in $\Z[X]$ have all modulus $1$, they must be roots of unity of degree at most $d$. Hence $g^{d!}$ must be a unipotent element, i.e. $(g^{d!}-1)^d=0$. However $\mathcal{V}:=\{g \in SL_d; g^{d!} \textnormal{ is unipotent }\}$ is a proper algebraic subvariety of $SL_d$, and hence Theorem \ref{subvariety} tells us that this set is exponentially small and can be ignored. It follows that
$$\mu^n(\mathcal{P}) \leq \sum_{m\leq C(S) n} \mu^n\{\mathcal{P}_m\} + O(e^{-cn})$$
where $\mathcal{P}_m$ is the set of $m$-powers. We will then apply the group sieve lemma to each $\mathcal{P}_m$ separately.

Now given $m\geq 2$, how many $m$-powers are there in $SL_d(\F_p)$ ? If $m$ is co-prime to the order of $SL_d(\F_p)$, then
every element is an $m$-power. So we wish to choose $p$ in such a way that there are not too many $m$-powers. For example, assume that $p \equiv 1$ mod $m$, so that $m$ divides the order of the multiplicative group of $\F_p$, which is a cyclic group of order $p-1$. In $\Z/(p-1)\Z$ there are precisely $\frac{p-1}{m}$ multiples of $m$. So there are exactly $(\frac{p-1}{m})^{d-1}$ $m$-powers in the subgroup of $SL_d(\F_p)$ made of diagonal matrices, which is a subgroup isomorphic to $(\Z/(p-1)\Z)^{d-1}$. In particular at least $\frac{(p-1)^{d-1}}{2}$ of the diagonal matrices are not $m$-powers. Among them at most $(p-1)^{d-2}$ have two identical diagonal entries, i.e. at least $\frac{(p-1)^{d-1}}{3}$ of them (for $p$ large) have distinct eigenvalues and thus a centralizer which is as small as possible, that is equal to the full diagonal group. In each conjugacy class of such a diagonal matrix, there are no more than $d!$ other such matrices. Taking the union of the conjugacy classes of these elements thus yields at least $|SL_d(\F_p)|/3d!$ different elements that are not $m$-powers. Thus we have shown that for large $p$ and any $m \geq 2$ with $p \equiv 1$ mod $m$

$$|\{m\textnormal{-powers in } SL_d(\F_p)| \leq (1-\frac{1}{3d!})|SL_d(\F_p)|$$

 To apply Lemma \ref{group-sieve} need now choose a sequence of distinct primes $\{p_i\}_{i=1,\ldots,N}$ with $N$ of exponential size in $n$. We choose one sequence of primes for each $m \leq Cn$. Dirichlet's theorem ensures that there are infinitely many primes congruent to $1$ mod $m$. More follows from the proof: there is in fact a positive density of such primes among the primes. 
However we need a uniform estimate as $m$ is allowed to vary from $1$ to $n$, while the primes we sieve with will be of exponential size in $n$. We need that there are exponentially many primes of exponential size congruent to $1$ mod $m$ uniformly in $m \leq C n$. So one needs a fairly precise quantitative version of Dirichlet's theorem: we need to know that the number $\pi(x;m,1)$ of primes congruent to $1$ mod $m$ and less than $x$ is at least say $\sqrt{x}$ \emph{uniformly} over all moduli $m \leq \log x$. The Siegel-Walfisz theorem says that the prime number theorem in arithmetic progressions is accurate uniformly for values $m$ going up to $(\log x)^A$ for any given $A\geq 1$. But it is non-effective in the sense that the first $x$ for which the estimate begins to be meaningful is not explicitly computable in terms of $A$ due to the possible presence of Siegel zeros. In our case, we need only a much weaker lower bound on the number of such primes and the estimate
$$\pi(x;m,1) = \frac{x}{\phi(m)}(1 + O(e^{-O((\log x)^{1/5})}))$$
holds uniformly for all $m\leq (\log x)^{3/2}$ with effective implied constants in the big $O$'s, where $\phi(m)$ denotes the Euler function (see $(7)$ on page 123 of Davenport's book \cite{davenport}). In particular $\pi(x;m,1) \geq \sqrt{x}$ for all $m \leq (\log x)^{3/2}$ and $x$ large enough.

We can now finish the proof of Theorem \ref{properpowers} (in our special case of Zariski-dense subgroups of $SL_d(\Z)$). Let $B=B(\eps,D,\alpha)>0$ be the constant from the group sieve lemma (Lemma \ref{group-sieve}). Set $\alpha=1/3d!$, $D=2d^2$, and $\eps=\eps(S)>0$ is given by the super-strong approximation theorem (Theorem \ref{severalp} for $r=2$ primes). Given a large $n$, and some $m\leq C(S)n$, by the above there are at least $\sqrt{x}$ distinct primes congruent to $1$ mod $m$ and smaller than $x:=e^{2n/B}$. Pick a subset of roughly $N = e^{n/B}$ of them, and apply Lemma \ref{group-sieve} to conclude that
$$\mu^n(\mathcal{P}_m) \leq e^{-n/B}$$
for each $m \leq C(S)n$. The result follows.
\end{proof}

\noindent \emph{Remark.} In the proof we used an effective version of the prime number theorem in progressions as opposed to the Siegel-Walfisz theorem, which is non-effective. This has only some sense if all other constants involved are indeed effective. The expander constant $\eps>0$ depends on the approximate subgroup constant $\delta$ from Proposition \ref{machine}. It is effective since all the algebraic geometry bounds used in Section \ref{approxsec} are effective, although not really explicit (see in particular \cite{pyber-szabo} where an attempt has been made to make some of these constants more explicit). Finally the first prime starting from which the super-strong approximation theorem holds is also effective as it relies on Nori's theorem (see the appendix of \cite{salehi-golsefidy-varju}) although far from explicit. So it is fair to say that the rate of exponential decay in Theorem \ref{properpowers}, though effective, is far from explicit.

\subsection{The generic Galois group is the Weyl group}

Given a matrix in $\SL_d(\Z)$, one may look at its characteristic polynomial and ask if it is irreducible over $\Q$. This amounts to say that the Galois group of the polynomial acts transitively on the roots. More generally when is the Galois group equal to the full group of all permutations of the roots ? when is it only a proper subgroup ?

Prasad and Rapinchuk \cite{prasad-rapinchuk-MRL} have shown that given a Zariski-dense subgroup $\Gamma$ of $\SL_d(\Z)$, the subset of elements in $\Gamma$ whose characteristic polynomial is irreducible, or even has full Galois group, is itself Zariski-dense in $\Gamma$, and even contains an entire coset of a certain finite index subgroup (see \cite[Remark 6]{prasad-rapinchuk-MRL}). They proved their result in a much greater generality (for an arbitrary semisimple group) and we refer the reader to \cite{prasad-rapinchuk-IHES} and to the excellent surveys \cite{prasad-rapinchuk-AMS} and \cite[\S 9]{prasad-rapinchuk-MSRI} for a description of their work and several further interesting results on how to find many elements in $\Gamma$ with various constraints on their characteristic polynomial.

Their method is also based on the study of the mod $p$ quotients of $\Gamma$. By Jordan's lemma (see below Lemma \ref{jlem}), the Galois group is maximal if and only if it has elements from every conjugacy class of the symmetric group. It is thus enough to find one prime number per conjugacy class for which the associated Frobenius element modulo $p$ is in that conjugacy class.

The same idea, this time combined with the group sieve lemma (Lemma \ref{group-sieve}) and the super-strong approximation theorem (Theorem \ref{ssapp}), can be applied to show the following somewhat stronger result, due to Jouve, Kowalski and Zywina \cite{jouve-kowalski}, which asserts that, the set of elements in $\Gamma$ whose characteristic polynomial is not all of $\frak{S}_d$ is exponentially small in the above sense of random walks: the probability that a random walk at time $n$ hits this subset decays exponentially with $n$. Note that combined with Theorem \ref{subvariety}, this also implies that the subset of elements in $\Gamma$ with full Galois group is Zariski-dense.

\begin{theorem}\label{sl}Let $d \geq 2$ and $\Gamma=\langle S \rangle \leq SL_d(\Z)$ be a Zariski-dense subgroup of $SL_d$. Let as above $\mu_S$ denote the uniform probability measure on the symmetric set $S$. Then there is $c=c(S)>0$ such that for all $n \geq 1$,
$$\mu_S^n(\{ \gamma \in \Gamma, Gal(\gamma) \neq \frak{S}_{d}\}) \leq e^{-c n }.$$
 Here $Gal(\gamma)$ denotes the Galois group of the extension $K_\gamma | \Q$, where $K_\gamma$ is the splitting field of the characteristic polynomial of $\gamma$ and $\frak{S}_d$ denotes the symmetric group of all permutations of $d$ elements. In particular
$$\mu_S^n(\{ \gamma \in \Gamma, \pi_\gamma \textnormal{ not }\Q\textnormal{-irreducible}\}) \leq e^{-c n }$$
\end{theorem}

We also refer the reader to the earlier work of Rivin \cite{rivin,rivin-zariski} for related statements and generalizations to other geometric contexts. And to the subsequent work of Gorodnik and Nevo \cite{gorodnik-nevo}, which proves a similar result (for arithmetic groups only) when counting with respect to a height function of $M_d(\Z)$ instead of the random walk average.

Theorem \ref{sl}  was proved by Jouve, Kowalski and Zywina \cite{jouve-kowalski} in the special case when $\Gamma$ has finite index in $SL_d(\Z)$.  When \cite{jouve-kowalski} was written the super-strong approximation theorem was still in limbo. Now that we have Theorems \ref{ssapp} and \ref{severalp} at our disposal, we can use them  in the argument from \cite{jouve-kowalski} and the whole proof goes through verbatim yielding Theorem \ref{sl} above. We give below the complete proof (see also \cite{lubotzky-rosenzweig}).

Jouve, Kowalski and Zywina proved their result in the wider generality of arithmetic subgroups of arbitrary connected semisimple groups (see below). Likewise, combined with the super-strong approximation, their argument extends to all Zariski-dense subgroups. It remains an open problem however to extend the Gorodnik-Nevo result to Zariski-dense subgroups.

In \cite{lubotzky-rosenzweig} Lubotzky and Rosenzweig extended these results to cover also non-connected semisimple algebraic groups and showed the interesting phenomenon that each coset of the connected component has its own generic Galois group, which may be different from the Weyl group of the connected component.

\bigskip
We now pass to the proof of Theorem \ref{sl}.

\bigskip

\begin{proof} The method is based on the following classical lemma of Jordan:

\begin{lemma}(Jordan)\label{jlem} Let $G$ be a finite group and $H$ a subgroup. If $H$ is a proper subgroup of $G$, then some conjugacy class of $G$ is disjoint from $H$.
\end{lemma}

In other words, the only subgroup of $G$ intersecting every conjugacy class is $G$ itself. Looking at the action by left translations on the set of left cosets $G/H$, we see that the lemma is equivalent to the following assertion: every transitive subgroup of $\frak{S}_d$ ($d \geq 2$) must contain a permutation with no fix points. For the proof of this simple lemma and a number of pretty applications to number theory, we refer the reader to Serre's short note \cite{serre-jordan-note}.

We will apply this lemma with $G=\frak{S}_d$ and $H=Gal(\gamma)$. Set $\mathcal{Z}:=\{\gamma \in \Gamma; Gal(\gamma) \neq \frak{S}_d\}$ and $\mathcal{Z}_C:=\{\gamma \in \Gamma; Gal(\gamma) \cap C = \varnothing\}$, where $C$ denotes a conjugacy class in the symmetric group $\frak{S}_d$. A conjugacy class $C$ of $\frak{S}_d$ is given by a partition of $d$ as $d=d_1+\ldots+d_k$ for integers $d_i \geq 1$. Jordan's lemma then tell us that
$$\mathcal{Z}= \bigcup_C \mathcal{Z}_C,$$
where the union ranges over all conjugacy classes of $\frak{S}_d$. Thus for proving Theorem \ref{sl} it will suffice to show that each $\mathcal{Z}_C$ is exponentially small. We will apply the group sieve lemma (Lemma \ref{group-sieve} above) to show precisely this.

\bigskip

As is well-known, to every prime $p$ not dividing the discriminant of $\pi_\gamma$, one can associate a particular conjugacy class of $Gal(\gamma)$, the Frobenius conjugacy class $Frob_p(\pi_\gamma)$. The prime ideals above $p$ in the splitting field $K_\gamma$ are permuted transitively by $Gal(\gamma)$. Each stabilizer subgroup is in bijection with the Galois group of the reduced polynomial $\pi_\gamma$ mod $p$ in $\F_p[X]$, which is a cyclic group generated by the Frobenius element mapping $x$ to $x^p$ in the corresponding residue field extension $\F_p[X]/(\pi_\gamma \textnormal{ mod }p)$.
 The corresponding elements in each stabilizer (decomposition) subgroup form the conjugacy class $Frob_p(\pi_\gamma)$ in $Gal(\gamma)$. The Frobenius element permutes the roots of $\pi_\gamma$ mod $p$ and its decomposition into a product of disjoint cycles corresponds to the factorization
$$ \pi_{\gamma \textnormal{ mod } p} = \pi_\gamma \textnormal{ mod } p  = P_1 \cdot \ldots \cdot P_k$$
into irreducible polynomials in $\F_p[X]$ with one cycle of length $\deg(P_i)$ for each $i=1,\ldots,k$. It determines a conjugacy class of $\frak{S}_d$ identified by the partition of $d$ given by $d=\deg(P_1)+\ldots+\deg(P_k)$.

\bigskip

Let $C$ be a conjugacy class of $\frak{S}_d$ determined by a partition $d=d_1+\ldots +d_k$ of $d$. From the above discussion, we see that if  $\gamma \in \mathcal{Z}_C$  and $p$ is a prime, then either the discriminant of $\pi_\gamma$ is divisible by $p$ and $\gamma$ mod $p$ has a multiple eigenvalue, or $\gamma$ mod $p$ is contained in the set of elements $g \in SL_d(\F_p)$ whose characteristic polynomial is without multiple roots (i.e. $g$ is regular semisimple) and whose factorization into irreducible polynomials in $\F_p[X]$ determines a partition of $d$ different from the partition associated to $C$.

The set of elements with a multiple eigenvalue (i.e. non regular semisimple elements) forms a proper subvariety of $SL_d$ of bounded degree (it is defined by the vanishing of the $\gcd$ of the characteristic polynomial and its derivative). The Lang-Weil bound, or the easier Schwarz-Zippel estimate (see \cite{bggt}), allows us to assert that this set has size $O(p^{d^2-2})$, while $SL_d(\F_p)$ has size at least $\Omega(p^{d^2-1})$, and is thus negligible. Consider now the second set.

\bigskip

To apply the group sieve lemma (Lemma \ref{group-sieve}) to the set $\mathcal{Z}_C $, it thus remains to show a uniform upper bound on the proportion of $SL_d(\F_p)$ the set of such elements can occupy. Or, equivalently, to prove a uniform lower bound on the size of the set  $\Omega_{p,C}$ of regular semisimple elements in $SL_d(\F_p)$ whose characteristic polynomial admits a factorization of the form dictated by the partition of $d$ associated to $C$.

It is easy to obtain such a lower bound. Every monic polynomial with constant term $(-1)^d$ is the characteristic polynomial of some matrix in $SL_d(\F_p)$, e.g. the companion matrix of the polynomial. So, given $C$, just pick a polynomial whose irreducible factors are distinct and whose degrees $d_i$'s are such that $d=d_1+\ldots+d_k$ is the partition associated to $C$. Let $g$ be the associated companion matrix. It belongs to $\Omega_{p,C}$ and so do all its conjugates.  It is a regular semisimple element of $SL_d(\F_p)$ and thus it belongs to a unique maximal torus $T$. All other regular semisimple elements in $T$ have the same associated partition of $d$, because they generate the same commutative subalgebra of matrices over $\F_p$. It follows that $\Omega_{p,C}$ contains $\cup_{g \in SL_d(\F_p)} gT^{reg} g^{-1}$, where $T^{reg}$ denotes the subset of regular elements in $T$ (i.e. with distinct eigenvalues). Hence

$$|\Omega_{p,C}| \geq \frac{|SL_d(\F_p)|}{|N(T)/T|)} - |\{g \in SL_d(\F_p) ; g \textnormal{ not regular semisimple }\}|,$$
where $N(T)$ is the normalizer of $T$. Now $N(T)/T$ is the Weyl group of $SL_d$, thus isomorphic to $\frak{S}_d$. As already mentioned the set of non regular semisimple elements in $SL_d(\F_p)$ is negligible (being of size $O(|SL_d(\F_p)|/p)$). Hence $|\Omega_{p,C}| \geq \frac{1}{2d!}|SL_d(\F_p)|$ say when $p$ is large enough.

\bigskip

To conclude the proof of Theorem \ref{sl}, it remains to apply the group sieve lemma (Lemma \ref{group-sieve}) to the sets $\mathcal{Z}_C$ for each conjugacy class $C$ of $\frak{S}_d$ and to the group $\Gamma$ with projection homomorphisms $\pi_i$ given by the reduction modulo $N$ primes $p_i$ of size at most  $N^2$ say, where $N$ is chosen of size $e^{n/B}$, with $B=B(\eps,D,\alpha)>0$ is the constant given by Lemma \ref{group-sieve} with $D:=3d^2$, $\alpha:= 1/2d!$ say, and $\eps=\eps(S)>0$ is given by the super-strong approximation theorem for two primes (Theorem \ref{severalp}). This ends the proof.

\end{proof}

In their paper Jouve, Kowalski and Zywina prove (the correct modified version of) Theorem \ref{sl} in the more general setting where the ambient group is a connected semisimple algebraic group defined (and not necessarily split) over a number field. Again while they treated only arithmetic subgroups, because the super-strong approximation theorem was not available to them, their method extends and applies to all Zariski dense subgroups. This was worked out by Lubotzky and Rosenzweig \cite{lubotzky-rosenzweig}, who also described in full the most general situation, when the field of definition is only assumed to be finitely generated over $\Q$ and, most interestingly, the algebraic group may not be connected nor semisimple. Without reaching out for the greatest generality, we will only state their theorem for split connected semisimple groups defined over a field of characteristic zero. In order to do so we first give some background on the Galois action on tori (see also \cite{prasad-rapinchuk-MSRI}, \cite{jouve-kowalski}).

\bigskip

Let the ambient group $\G$ be a connected semisimple algebraic group defined and split over some finitely generated field $K$ of characteristic zero. This means that $\G$ admits a maximal torus $T_0$ which is defined and diagonalizable in any linear representation of $\G$ over $K$. To every regular semisimple element $g$ in $\G(K)$ corresponds the unique maximal $K$-torus $T_g$ it contains. A priori $T_g$ is not diagonalizable over $K$, but there is a smallest finite extension of $K$, the splitting field $K_{T_g}$ of $T_g$ such that $T_g$ is conjugate over $K_{T_g}$ to the $K$-split (i.e. diagonalizable) torus $T_0$. The Galois group $Gal(g)$ of the extension $K_{T_g}|K$ acts on the group $X(T_g)$ of characters of $T_g$. The group $X(T_g)$ is the free abelian group of rank $r=rank(\G)$ made of algebraic homomorphisms from $T_g$ to the multiplicative group $\G_m$. The Galois action of $Gal(g)$ on $X_{T_g}$ is via the formula $$\sigma(\chi(t)) = \prescript{\sigma}{}\chi ( \sigma(t)).$$
This action is faithful and thus $Gal(g)$ can be viewed as a finite subgroup of $Aut(X(T_g)) \simeq \GL_r(\Z)$.

 The Weyl group $W(T_g):=N(T_g)/Z(T_g)$ of $T_g$, where $N(T_g)$ is the normalizer and $Z(T_g) = T_g$ the centralizer of $T_g$, can also be viewed as a subgroup of $Aut(X(T_g))$ using the action by conjugation of the normalizer $N(T_g)$, namely $$\chi \mapsto (t \mapsto \chi(n^{-1}tn)),$$ for $n \in N(T_g)$ and $t \in T_g$.

Under the identification, it turns out that $Gal(g)$ becomes a subgroup of the Weyl group $W(T_g)$:  indeed fixing a $K$-split maximal torus $T_0$, there is an element $x \in \G(\overline{K})$ such that $T_g=xT_0x^{-1}$, because all maximal tori are conjugate over the algebraic closure $\overline{K}$ of $K$. Now from the fact that $T_g$ is defined over $K$, we see that $n_\sigma:=\sigma(x)x^{-1}$ belongs to $N(T_g)$, and that $\prescript{\sigma}{}\chi(t)= \chi(n_\sigma^{-1} t n_\sigma)$ for all $t \in T_g$. Recall that the isomorphism class of $W(T)$ is independent of $T$, it is the Weyl group $W_\G$ of $\G$. When $\G=SL_d$, then $W_\G \simeq \frak{S}_d$.

We can now state the theorem of Jouve, Kowalski and Zywina \cite{jouve-kowalski} in the version proved by Lubotzky and Rosenzweig \cite{lubotzky-rosenzweig} (i.e. for Zariski-dense subgroups over  fields of characteristic zero and not merely arithmetic groups over number fields).

\begin{theorem}\label{galois} Let $\G$ be a connected semisimple algebraic group defined and split over $K$, a finitely generated field extension of $\Q$. Suppose $\Gamma \leq \G(K)$ is a Zariski-dense subgroup and $\mu$ a symmetric probability measure whose support is a finite generating subset of $\Gamma$. Then there is $c>0$ such that
$$\mu^n( \gamma \in \Gamma ;  Gal(\gamma) \lneq W(T_\gamma)) \leq e^{-cn}.$$
\end{theorem}

Here again, this implies (via Theorem \ref{subvariety}) that the set of elements $\gamma$ with $Gal(\gamma)=W(T_\gamma)$ is Zariski-dense in $\Gamma$, a fact first established by Prasad and Rapinchuk in \cite{prasad-rapinchuk-MRL}.

The proof of Theorem \ref{galois} follows the same sieving argument as in the special case of subgroups of $SL_d(\Z)$ presented above. Using a specialization argument Lubotzky and Rosenzweig reduce to the case when $K$ is a number field. Then the group sieve lemma together with the super-strong approximation theorem (applied to the reduction of scalars of $\G$ from $K$ to $\Q$, see Remark \ref{numberfields}) apply in a similar way.

If $\G$ is not split over the base field $K$, or if it is not connected, then the theorem still holds, but the generic Galois group of an element $\gamma$ may no longer be the Weyl group (in the connected non split case, the Weyl group appears only as a subgroup) and it will depend (only) on the coset of the connected component of $\G$ it lives in. See \cite{prasad-rapinchuk-MSRI}, \cite{jouve-kowalski} and \cite{lubotzky-rosenzweig} for this and further information about the generic $Gal(\gamma)$.

\bigskip

\noindent \emph{Acknowledgements.} It is a pleasure to thank Florent Jouve, Alex Lubotzky, Laci Pyber, Andrei Rapinchuk, Lior Rosenzweig, Peter Sarnak and Terry Tao for their comments on an earlier version of this article. I am also grateful to Edmund Robertson and Colin Campbell for their patience and the gentle reminders that helped me finish this article.

\appendix

\bibliographystyle{abbrv}
\bibliography{bibfile}

\end{document}